\documentclass[preprint,10pt]{amsart}
\usepackage{amssymb,latexsym,amscd, amsthm, epsfig,amsmath,amssymb}
\usepackage{textcomp}
\usepackage{color}
\usepackage{xypic}
\usepackage[all]{xy}
\usepackage[T1]{fontenc}
\usepackage[latin1]{inputenc}
\usepackage[all]{xy}
\usepackage{verbatim}
\usepackage{amsfonts,amsthm}
\usepackage[twoside=false,left=3.3cm,right=3.3cm,top=3cm,bottom=3cm]{geometry}
\usepackage{cancel}
\usepackage{graphicx}
\usepackage{hyperref}
\usepackage{pigpen} 
\usepackage{thmtools}
\usepackage{thm-restate}
\usepackage{cleveref}
\usepackage[
textwidth=3cm,
textsize=small,
colorinlistoftodos]
{todonotes}
\usepackage{graphicx}
\graphicspath{ {images/} }

\usepackage{url}
 \newtheorem{thm}{Theorem}[section]

 \theoremstyle{definition}
 \newtheorem{dfnt}[thm]{Definition}
  
   \newtheorem{construction}[thm]{Construction}
    \newtheorem{notation}[thm]{Notation}

    \newtheorem{ex}[thm]{Example}  
       
        \newtheorem{cor}[thm]{Corollary}
 \newtheorem{lem}[thm]{Lemma}
 \newtheorem{prop}[thm]{Proposition}
      \newtheorem{teorema}[thm]{Theorem}
      \newtheorem{rem}[thm]{Remark}
        
\newtheoremstyle{theorem}{\topsep}{\topsep}
{\slshape}{}{\bf}{{\normalfont.}}{.5em}{}
\usepackage{soul}

\usepackage{xcolor}

\newsavebox\MBox

\newcommand{\po}{\ar@{}[dr]|{\text{\pigpenfont R}}}
\newcommand{\pb}{\ar@{}[dr]|{\text{\pigpenfont J}}}

\def\pr{\ensuremath{\mathop{\textrm{\normalfont pr}}}}

\def\im{\ensuremath{\mathop{\textrm{\normalfont im}}}}

\def\ker{\ensuremath{\mathop{\textrm{\normalfont ker}}}}
\def\coker{\ensuremath{\mathop{\textrm{\normalfont coker}}}}
\def\id{\ensuremath{\mathop{\textrm{\normalfont id}}}}

\title{Characteristic classes as complete obstructions}

\author[M. Rovelli]{Martina Rovelli}
\address[M. Rovelli]{Department of Mathematics, EPF Lausanne, Switzerland}
\email{martina.rovelli@alumni.epfl.ch}
\begin{document}
\maketitle
\setstcolor{blue}
\begin{abstract}
In the first part of this paper, we propose a uniform interpretation of characteristic classes as obstructions to the reduction of the structure group and to the existence of an equivariant extension of a certain homomorphism defined a priori only on a single fiber of the bundle. By plugging in the correct parameters, we recover several classical theorems. Afterwards, we define a family of invariants of principal bundles that detect the number of group reductions that a principal bundle admits. We prove that they fit into a long exact sequence of abelian groups, together with the cohomology of the base space and the cohomology of the classifying space of the structure group.
\end{abstract}

\tableofcontents

\section*{Introduction}
By a classical result \cite{steenrod, dold}, every principal $G$-bundle $E$ over a CW-complex $X$ is classified up to equivalence by the homotopy class of a map $X\to BG$, where $BG$ is the classifying space of the group.
Having fixed an abelian group of coefficents $Z$, one can look at $
\chi_k\colon H^k(BG;Z)\to H^k(X;Z)
$, the map induced in cohomology in degree $k$.
The elements of $H^k(BG;Z)$ are referred to as \emph{universal characteristic classes}, and depend only on the topological group $G$. On the other hand, their images in the cohomology of the base space $H^k(X;Z)$ are invariants of the bundle $E$, known as \emph{characteristic classes} \cite{ms}.

\vphantom{}

While the different flavors of characteristic classes are often studied separately, in this paper we develop a uniform treatment that can be applied in great generality. We first prove that every characteristic class measures the obstruction to two features of a principal bundle: the possibility of reducing the structure group, and the possibility of extending an interesting quantity to the whole bundle (being a priori only defined on a single fibre). The general approach gives a global explanation of the fact that so many examples of characteristic classes were already known to measure obstruction to group reduction. Afterwards, we define the \emph{plus-cohomology groups} of a principal bundle: invariants that detect the number of interesting group reductions that the bundle admits.
We prove that the plus-cohomology groups fit into a long exact sequence together with the cohomology groups of the base space and of the classifying space of the structure group.

\vphantom{}

For some values of the parameters, it is known that characteristic classes measure the obstruction to certain geometric features of the principal bundle. For example, when $k=1$, $Z=\mathbb Z/2$ and $G=O(n)$, the generator of the group $H^1(BO(n);\mathbb Z/2)\cong\mathbb Z/2$ is the \emph{first Stiefel-Whitney class} $w_1\in H^1(BO(n);\mathbb Z/2)$, which measures obstruction to orientability. The orientability of an $O(n)$-bundle can be expressed by saying that the structure group can be reduced to $SO(n)$ or, equivalently, by saying that the choice of a sign on a distinguished fiber can be extended coherently to the whole bundle. This is encoded precisely in the following theorem \cite{ms}.\\
\ \\
\textbf{Classical Theorem.}
\label{orientability}
For $E$ an $O(n)$-bundle over a CW-complex $X$,
the following are equivalent.
\begin{itemize}
	\item[(A)] There exists an $SO(n)$-bundle $\hat E$ over $X$ such that $\hat E\otimes_{SO(n)}O(n)\cong E$.
	\item[(B)] The first Stiefel-Whitney class $w_1$ of $E$ vanishes, i.e.,
$w_1(E)=0\in H^1(X;\mathbb Z/2)$.
	\item[(C)] The map $\det\colon O(n)\to O(1)\cong\mathbb Z/2$ extends to an $O(n)$-equivariant map $E\to\mathbb Z/2$.
	\end{itemize}

\vphantom{}

The first contribution of this paper is an analog of this theorem for any universal characteristic class $c\in H^k(BG;Z)$, where $G$ is any topological group whose homotopy groups are countable (e.g. a Lie group) and $Z$ is a countable abelian group. To this end, we first construct a topological group $\hat G(c)$ and a homomorphism of topological groups $\gamma\colon G^c:=\tilde\Omega BG\to M^{k-1}Z$, which will play for $c\in H^k(BG;Z)$ the roles played for the first Stiefel-Whitney class $w_1\in H^1(X;\mathbb Z/2)$ in the theorem above by $SO(n)$ and by the determinant map, respectively. These constructions rely on Milnor's model $\tilde\Omega X$ of the based loop space for a countable CW-complex $X$ from \cite{milnor1} and on Milgram's model of classifying space $MA$ of an abelian topological group $A$ from \cite{milgram}.

\vphantom{}

While the precise definitions of $\gamma$ and $\hat G(c)$ require some work, their homotopical behaviour can be described as follows.
If we regard the universal characteristic class $c$ as a map $c\colon BG\to K(k;Z)$, then $\gamma$ is a homomorphism of topological groups homotopy equivalent to $\Omega c\colon G\simeq \Omega BG\to K(k-1;Z)$, and $\hat G(c)$ is a topological group which gives a model for the homotopy fiber of $\Omega c$. In particular, the homotopy groups of $\hat G(c)$ are the same as $G$, except possibly in degrees $k-2$ and $k-1$, where the values depend on $c$ according to an exact sequence of the following form:
$$\xymatrix{0\ar[r]&\pi_{k-1}(\hat G(c))\ar[r]^-{}&\pi_{k-1}(G)\ar[r]&Z\ar[r]&\pi_{k-2}(\hat G(c))\ar[r]^-{}&\pi_{k-2}(G)\ar[r]&0.}$$

Although the reduction group $\hat G(c)$ and the looping homomorphism $\gamma$ are only unique in a homotopical sense, which is explained after the constructions, they succeed at detecting the vanishing of the characteristic class $c$ for $G$-bundles. This idea is summarized by the following result, that is a combination of Theorems \ref{groupreductioncor} and \ref{retractioncor}.
\\
\ \\
\textbf{Main Result 1.}
Let $G$ be a topological group whose homotopy groups are countable, $Z$ a countable abelian group, $k>1$ and $c\in H^k(BG;Z)$ a universal characteristic class. For $E$ a $G$-bundle over a CW-complex $X$,
the following are equivalent.
\begin{itemize}
	\item[(A)] There exists a $\hat G(c)$-bundle $\hat E$ over $X$ such that $\hat E\otimes_{\hat G(c)}G\cong E$.
	\item[(B)] The characteristic class $c$ of $E$ vanishes, i.e., $c(E)=0\in H^k(X;Z)$.
	\item[(C)] The homomorphism of topological groups $\gamma\colon G^{\text{c}}\to M^{k-1}Z$ extends to a $G^{\text{c}}$-equivariant continuous map $E^{\text{c}}\to M^{k-1}Z$.
\end{itemize}
Here, $G^{\text{c}}$ and $E^{\text{c}}$ are \emph{fattened-up} versions of $G$ and $E$, which are introduced in Construction \ref{cofibrantreplacementbundle} and do not change the homotopy type of $G$ and $E$, and $M^{k-1}Z$ is the $(k-1)$-fold Milgram delooping from \cite{milgram}: an Eilenberg-Maclane space of type $K(Z;k-1)$ that is also a group.

\vphantom{}

The equivalence between statements of type (A) and (B) is well-known for groups appearing in the Whitehead tower of the orthogonal group $O(n)$ and certain characteristic classes. 
For instance, the second Stiefel-Whitney class measures the obstruction to admitting a spin structure, the first fractional Pontryagin class measures the obstruction to admitting a string structure, and the second fractional Pontryagin class measures the obstruction to admitting a Fivebrane structure \cite{sss}.
These cases are recovered, and strengthened by the fact that the reduction of the structure group is expressed not only by a lift of the classifying map, but by an honest tensor product with a topological group.
The general construction $\hat G(c)$ might still look different from some of the explicit groups that are known for specific examples. However, there is a homomorphism of topological groups from $\hat G(c)$ to the classical model, which is also a weak equivalence. For instance, in the case of the first Stiefel-Whitney class, there is a map of topological groups $\hat G(w_1)\to SO(n)$ that is also a weak equivalence. As a further application, the construction $\hat G(c)$ can also be used to build a Whitehead tower \cite{whiteheadtower} of $G$, where every layer is a topological group (and the maps homomorphisms of such).

\vphantom{}

On the other hand, for an arbitrary universal characteristic class $c\in H^k(BG;Z)$, the equivalence between conditions analogous to (B) and (C) is not as evident in the literature. Even for well-studied examples of characteristic classes (e.g. Chern classes), there is no obvious choice of homomorphism for which the characteristic class measures the obstruction to a coherent extension.

\vphantom{}

While the first part of the paper explains how characteristic classes are related to the \emph{existence} of group reductions of a fixed bundle, one can then wonder about the \emph{uniqueness} of such group reductions, or more generally look for quantitative analogs of the first result. For instance, when $c:=w_2\in H^2(BSO(n);\mathbb Z/2)$ is the second Stiefel-Whitney class, the problem translates to counting how many spin structures (up to a suitable equivalence relation) an orientable bundle admits. In the second part of the paper we develop some tools in order to answer these kinds of questions.
More precisely, we give a quantitative explanation of the equivalence between (B) and (C) for a universal characteristic class $c\in H^k(BG;Z)$ represented by a homomorphism of topological groups $\gamma$.
Indeed, given a $G$-bundle $E$ over $X$, the vanishing of the characteristic class $c(E)$ can be expressed by saying that $c$ is in the kernel of the characteristic map $
\chi_k\colon H^k(BG;Z)\to H^k(X;Z)
$ of $E$, whereas the extension of the homomorphism $\gamma$ can be expressed by saying that $\gamma$ is in the image of a certain map. Compactly, this phenomenon is governed by an exact sequence of abelian groups. In order to describe the maps of such sequence, we need to describe the setup.

\vphantom{}

We first define a category ${\mathcal Bun_*}$ of pointed bundles that contains topological groups as bundles over a point, and pointed spaces as bundles with a trivial structure group.
In this category lives in particular the \emph{Nomura-Puppe sequence} (cf. \cite{fh}) of any $G$-bundle $E$ over $X$,
$$\xymatrix{\tilde\Omega X\ar[r]&G\ar[r]^-{\iota}&E\ar[r]^-{\pi}&X\ar[r]&BG.}
$$
This sequence is obtained by including the structure group $G$ as a distinguished fiber of $E$, by projecting $E$ onto the base space $X$, and by adding the usual classifying map $X\to BG$ and the \emph{dual} classifying map $\tilde\Omega X\to G$, which is the homomorphism of topological groups constructed in a previous paper \cite{rovelli}.
We explain that any topological abelian group, and in particular the Milgram delooping $M^kZ$, is an abelian group object in the category $\mathcal Bun_*$, and thus the homset ${\mathcal Bun_*}(E,M^kZ)$ becomes an abelian group, when endowed with pointwise multiplication. We define on it a compatible equivalence relation $\simeq_+$, and focus on the \emph{plus-cohomology group} of $E$:
$$H_+^k(E;Z):={\mathcal Bun_*}(E^{\text{c}},M^kZ)/_{\simeq_+}.$$
When the input is a discrete bundle $X$, namely a pointed space, or a codiscrete bundle $G$, namely a topological group, the plus-cohomology groups recover the ordinary cohomology groups (of the classifying space), i.e.,
$$H^k_+(X;Z)\cong H^k(X;Z)\text{ and }H^k_+(G;Z)\cong H^{k+1}(BG;Z).$$
More generally, the plus-cohomology of a trivial bundle $X\times G$ gives the sum of the two, and the plus-cohomology of a contractible bundle $EG$ is trivial, i.e.,
$$H^k_+(X\times G;Z)\cong H^k(X;Z)\oplus H^{k+1}(BG;Z)\text{ and }H^k_+(EG;Z)\cong0.$$
As the equivalence relation $\simeq_+$ is not totally well behaved with respect to precomposition of maps,
the assignment $E\mapsto H^k_+(E;Z)$ will not define a functor on the category ${\mathcal Bun_*}$. However, the assignment acts on the Nomura-Puppe sequence of a bundle $E$, which leads to the following result appearing later as Theorem \ref{longexactsequence}.\\
\ \\
\textbf{Main Result 2.}
Let $G$ be a topological group whose homotopy groups are countable. For every $G$-bundle $E$ over a CW-complex $X$, there is a long exact sequence of abelian groups
$$
\xymatrix{\dots\ar[r]&H^{k-1}_+(E;Z)\ar[r]&H^{k-1}_+(G;Z)\ar[r]^-{\chi_k}&H^k_+(X;Z)\ar[r]&H^k_+(E;Z)\ar[r]&\dots},
$$
of which the connecting map can be identified with the characteristic map $
\chi_k\colon H^k(BG;Z)\to H^k(X;Z)
$ induced by the classifying map of $E$.

\vphantom{}

From the long exact sequence, one can extract a lot of geometric information about the bundle. The exactness in $H^k(BG;Z)$ reflects the fact that characteristic classes are obstructions to the extension of the looping homomorphism to the whole bundle, incorporating part of the previous theorem. The exactness in $H^k_+(E;Z)$ says that the plus-cohomology group is a non-trivial extension of the kernel of the characteristic map in degree $k+1$, and the cokernel of the characteristic map in degree $k$:
$$\xymatrix{0\ar[r]&\coker(\chi_k)\ar[r]&H^k_+(E;Z)\ar[r]&\ker(\chi_{k+1})\ar[r]&0.}$$
They carry, respectively, information about the characteristic classes of the bundles in degree $k$ and $k+1$. This is an indication of some surprising interaction between characteristic classes in different degrees.

\vphantom{}

The existence of such a long exact sequence suggests the existence of (and is used in the author's Ph.D. thesis \cite{rovellitesi} to construct) an unexpected isomorphism
$$H_+^k(E;Z)\cong H^{k+1}(\textrm{cof}(t);Z)$$
between the plus-cohomology groups of $E$ and the ordinary cohomology groups of the homotopy cofiber $\textrm{cof}(t)$ of the classifying map $t$ of $E$. This yields an alternative definition of the plus-cohomology groups, which has a homotopical flavor and does not require any further assumption on the structure group $G$. On the other hand, the original definition is the one that allows us to extract the geometric information of the bundle. For instance, we explain how the plus-cohomology groups of a bundle can be used to enumerate the group reductions with respect to a fixed structure group, e.g., enumerate the spin structures of an orientable bundle or enumerate the string structures of a spin bundle.

\vphantom{}

In Section 1 we set up the framework for the structure group $G$ the theory applies to. We start by stating useful properties of countable CW-groups, which are very special topological groups. We then study topological groups that are equivalent to countable CW-groups in a suitable sense, and prove that they coincide with topological groups whose homotopy groups are countable. In Section 2 we construct from every universal characteristic class a reduction group and a looping homomorphism, and prove that they both detect the vanishing of that universal characteristic class. In Section 3 we define the plus-cohomology groups of principal bundles. We prove that they fit into a long exact sequence, and explain what features of a bundle they detect.

\ 

\textbf{Acknowledgements}.
I would like to thank Kathryn Hess for her support and for her constructive feedback on earlier versions of this paper. I am grateful to Jeffrey Carlson for suggesting an alternative description of the plus-cohomology groups, and to Viktoriya Ozornova for many extremely helpful discussions.

\section{Framework for the structure group}

\subsection{Recollection of prior results}
We refer the reader to \cite{steenrod,dold} for the general theory of principal bundles. In this section, we recall the preliminary results that are needed in the paper. Unless specified otherwise, in this paper all bundles are supposed to be principal bundles.

\begin{notation}
Let $X$ be a topological space and $G$ a topological group.
\begin{itemize}
\item A \textbf{$G$-bundle} over $X$ consists of a space $E$, together with a map $E\to X$ and a right $G$-action $E\times G\to E$, such that $E$ is \textbf{locally trivial}, i.e., there exists an open cover $\{U\}$ of $X$ and $G$-homeomorphisms $E|_U\cong U\times G$ compatible with the projection over $X$. The group $G$ is called \textbf{structure group}, and the space $X$ is called \textbf{base space}. We denote by ${}_X\mathcal Bun_G$ the class of principal $G$-bundles over $X$. Two $G$-bundles $E$ and $E'$ over $X$ are \textbf{equivalent} if there exists an \textbf{equivalence of bundles} $\Phi\colon E\to E'$, i.e., a map $G$-equivariant continuous map that respects the projection over $X$. Equivalence of bundles defines an equivalence relation, because the $G$-action guarantees that any equivalence is bijective with continuous inverse, and we write $E\cong E'$ if $E$ and $E'$ are equivalent.
\item If $a\colon G\to G'$ is a map of topological groups, it endows $G'$ with a $G$-action, and we denote by $a_*E$ the tensor product $E\otimes_GG'$, which is a $G'$-bundle over $X$ and comes with a canonical map $E\cong E\otimes_GG\to E\otimes^a_GG'=a_*E$. Dually, if $f\colon X'\to X$ is a continuous map, we denote by $f^*E$ the pullback $X'\times_XE$ of $E$ along $f$, which is a $G'$-bundle over $X'$ and comes with a canonical map $f^*E=X'\times_XE\to X\times_XE\cong E$. Given $f'\colon X'\to X$ another continuous map and $a'\colon G'\to G''$ another map of topological groups, there are canonical equivalences of $G''$-bundles over $X''$:
$$(f\circ f')^*(E)\cong f'^*f^*(E),\quad a_*f^*(E)\cong f^*a_*(E)\quad\text{and}\quad(a'\circ a)_*(E)\cong a'_*a_*(E).$$
\end{itemize}
\end{notation}

The following property is easy to verify.

\begin{prop}
\label{naturality}
Let $a\colon G\to G'$ be a map of topological groups and $f\colon X\to X'$ a continuous map. Let $E$ be a $G$-bundle over $X$ and $E'$ a $G'$-bundle over $X'$.
Any $a$-equivariant map $\psi\colon E\to E'$ that induces $f\colon X\to X'$ also induces an equivalence $a_*E\cong f^*E'$ of $G'$-bundles over $X$.
\hfill{$\square$}
\end{prop}

Recall that, if $E$ is a $G$-bundle over $B$ and $E$ is contractible, we say that $E$ is a \textbf{universal $G$-bundle} for $G$ and $B$ is a \textbf{classifying space} for $G$.

\begin{teorema}[{\cite[\textsection3]{milnor2}}]
\label{classifyingspacefunctorial}
Let $G$ be a topological group. There exists a universal $G$-bundle $EG$ with corresponding classifying space $BG$. The constructions $EG$ and $BG$ are functorial on topological groups.
\end{teorema}

We recall the universal property of universal $G$-bundles and classifying spaces, which determines the weak homotopy type of the classifying space of $G$.

\begin{teorema}[\cite{dold}]
\label{classificationGbundles}
Let $E$ be a universal $G$-bundle with base space $B$. For every CW-complex $X$, the pullback construction induces a natural bijection
$$(-)^*(E)\colon \mathcal Top_*(X,B)/_{\simeq}\ \cong\ _B\mathcal Bun_G/_{\cong}.$$
\end{teorema}

\begin{rem}
\label{classifyingspaceCWcomplex}
It is easy to see that any topological group $G$ admits a model of classifying space that is also a CW-complex. Indeed, given a weak equivalence $B\to BG$ with $B$ a CW-complex, the pullback $(B\to BG)^*(EG)$ is a universal $G$-bundle with base space the CW-complex $B$. However, it is often convenient to have at one's disposal models of classifying spaces or universal bundles that have further topological or algebraic properties.
\end{rem}

Recall from \cite{milnor1} that a \textbf{countable CW-complex} is a CW-complex with countably many attaching cells. Recall from \cite{milnor2} that a \textbf{countable CW-group} is a countable CW-complex with a structure of a topological group such that the multiplication and the inversion are cellular. Countable CW-groups should not be confused with \emph{$G$-CW-complexes} from equivariant homotopy theory
(for instance treated in \cite{illman,mayequivariant,ckv}).

For countable CW-groups, there are special models of classifying spaces that enjoy further properties.

\begin{teorema}[{\cite[\textsection5]{milnor2}}]
\label{niceclassifyingspace}
Let $G$ be a countable CW-group. There exists a universal $G$-bundle such that the total space $\tilde EG$ and the base space $\tilde BG$ are countable CW-complexes. The constructions $\tilde EG$ and $\tilde BG$ are functorial on countable CW-groups (and maps of topological groups between such).
\end{teorema}

\begin{rem}
\label{comparisonmodels}
Given a countable CW-group $G$, the spaces $EG$ and $\tilde EG$ from \cite{milnor2} are both defined. They share the same underlying $G$-set (given by the fat realization of the bar construction of $G$), but $\tilde EG$ is endowed with what Milnor calls the \emph{weak} topology (that is the usual topology on a fat geometric realization), whereas $EG$ is endowed with what he call the \emph{strong} topology (which is the initial topology with respect to a certain class of maps). Using the explicit descriptions of these topologies from \cite[Section 2]{milnor2}, one can see that the identity $\id\colon\tilde EG\to EG$ defines a continuous map.
The corresponding classifying spaces $BG$ and $\tilde BG$ also carry (a priori) different topologies. The identity map $\id\colon\tilde BG\to BG$ defines however a weak equivalence (because the identity $\id\colon\tilde EG\to EG$ is a weak equivalence and is compatible with the $G$-action). Both models are needed in this paper (e.g. in the proof of Proposition \ref{almostnicegroups}). On the one hand, $BG$ is always defined and used when $G$ is an arbitrary group. On the other hand, when restricting to the case of $G$ being a countable CW-group, $\tilde BG$ is defined and gives a model of classifying space that admits the structure of a countable CW-complex.
\end{rem}

The emphasis on countable CW-complexes is also due to the following theorem of Milnor.

\begin{teorema}[Milnor]
\label{classifyinggroup}
Let $X$ be a connected countable CW-complex. There exists a countable CW-group $\tilde\Omega X$ and a universal $\tilde\Omega X$-bundle $\tilde PX$ over $X$, with the property that $(-)_*(\tilde PX)$ induces for every topological group $G$ a surjection
$$(-)_*(\tilde PX)\colon \mathcal Gp(\mathcal Top)(\tilde\Omega X,G)\twoheadrightarrow{}_X\mathcal Bun_G/_{\cong}.$$
\end{teorema}

\begin{proof}
Suppose at first that $X$ is a countable connected simplicial complex. Then the constructions of the topological group $\tilde\Omega X$ and of the $\tilde\Omega X$-bundle $\tilde PX$ over $X$ are from \cite[Theorem 3.1]{milnor1}. Still in \cite[Theorem 3.1]{milnor1}, Milnor endows $\tilde\Omega X$ with the structure of a countable CW-complex, which makes $\tilde\Omega X$ into a countable CW-group, as mentioned in the first paragraph of \cite[Section 5]{milnor2}. Finally, the surjection of $(-)_*(\tilde PX)$ is from \cite[Theorem 5.1]{milnor1}.

When, more generally, $X$ is a connected countable CW-complex, Milnor explains in \cite[Corollary 3.7]{milnor1} how to obtain the topological group $\tilde\Omega X$ and the $\tilde\Omega X$-bundle $\tilde PX$. Namely, given a homotopy equivalence $s\colon X\to S$, where $S$ is a countable connected simplicial complex, one can take $\tilde\Omega X:=\tilde\Omega S$ and $\tilde PX:=s^*\tilde PS$. Since $s$ is a homotopy equivalence, the surjection of $(-)_*(\tilde PS)$ implies the surjection of $(-)_*(\tilde PX)$.
\end{proof}

Milnor's model of classifying space $\tilde B$ will be used to identify maps of topological groups, in the sense of the following definition.
\begin{dfnt}
\label{basedalgebraicequivalence}
Let $G$ and $G'$ be countable CW-groups. We say that two maps $a,b\colon G\to G'$ of topological group are \textbf{algebraically equivalent} if $\tilde Ba\simeq\tilde Bb\colon\tilde BG\to\tilde BG'.$ In this case we write $a\equiv b\colon G\to G'$.
\end{dfnt}

Algebraic equivalence is indeed an equivalence relation, and detects the kernel of the surjection above, as made precise by the following theorem from a previous paper.

\begin{teorema}[\cite{rovelli}]
\label{classificationnice}
Let $X$ be a connected countable CW complex and $G$ a countable CW-group. Then the assignments
$(-)_*(\tilde PX)$ and $(-)^*(\tilde EG)$
induce bijections
$$\mathcal Gp(\mathcal Top)(\tilde\Omega X,G)/_{\equiv}\ \cong\ _X\mathcal Bun_G/_{\cong}\ \cong\ \mathcal Top(X,\tilde BG)/_{\simeq}.$$
\end{teorema}

\begin{rem}
\label{adjunctionweakequivalence}
If a map of topological groups $a\colon\tilde\Omega X\to G$ and a continuous map $g\colon \tilde BG\to X$ correspond to each other as in Theorem \ref{classificationnice}, i.e., there is an equivalence $a_*\tilde PX\cong g^*\tilde EG$ of $G$-bundles over $X$, then $a$ is a weak equivalence if and only if $g$ is a weak equivalence.
Indeed, the map $\Phi$ induced by $a$ and $g$
$$\Phi\colon \tilde PX\cong\tilde PX\otimes_{\tilde\Omega X}\tilde\Omega X\to\tilde PX\otimes_{\tilde\Omega X}G=a_*\tilde PX\cong g^*\tilde EG=X\times_{\tilde BG}\tilde EG\to\tilde BG\times_{\tilde BG}\tilde EG\cong\tilde EG$$
is a weak equivalence (because its source and target are contractible), induces $g\colon X\to\tilde BG$ at the level of base spaces and restricts to $a\colon \tilde\Omega X\to G$ on every fiber. It follows that $a$ is a weak equivalence if and only if $g$ is one.
\end{rem}

We state one last property of countable CW-complexes that will be needed later. The proof is elementary but technical, and is therefore postponed until the appendix.

\begin{prop}
\label{totalspaceCWcomplex}
Let $E$ be a fiber bundle over $X$ with fiber $F$. If $X$ and $F$ admit the structure of countable CW-complexes, so does $E$.
\end{prop}

Many of the constructions that have been recalled involve countable CW-groups. It is in general hard to understand whether a given topological group endowed with a CW-structure is indeed a countable CW-group. We collect here the main known examples.

\begin{ex}
\begin{itemize}
\item All discrete countable groups are countable CW-groups.
\item For every connected countable CW-complex $X$, the group $\tilde\Omega X$ is a countable CW-group.
\item Finite products of countable CW-groups are countable CW-groups.
\item It is possible but unclear whether the usual matrix groups (e.g., the orthogonal group $O(n)$, the unitary group $U(n)$, and the simplectic group $Sp(n)$) are (countable) CW-groups. A finite CW-structure for $O(n)$, $U(n)$, and $Sp(n)$ (seen as degenerate Stiefel manifolds) is given, for instance, by the case $k=0$ of \cite[Theorem 2.1]{steenrodbook}. By reading through the definition of the cells in \cite[Definition 1.5]{steenrodbook} one could investigate by direct inspection whether the multiplication and the inversion maps are cellular. While it is easy to treat the case $n=1$, addressing the question for general $n$ becomes very quickly computationally involved.
\end{itemize}
\end{ex}

The reader should not be discouraged by the fact that the condition of being a countable CW-group is highly restrictive. Indeed, the next section is devoted to relaxing this notion in order to include the majority of relevant examples (e.g. all Lie groups).

\subsection{Groups that admit a countable CW-replacement}
Thanks to the following proposition, we see that many groups are at least equivalent to a countable CW-group in a precise sense. The theory developed in the later sections of this paper applies to such larger class of topological groups. Recall from Remark \ref{comparisonmodels} that for every countable CW-group $G$ the identity maps $\id\colon\tilde EG\to EG$ and $\id\colon\tilde BG\to BG$ are weak equivalences.

\begin{prop}
\label{almostnicegroups}
For a topological group $G$, the following are equivalent.
\begin{enumerate}
\item There exists a countable CW-group $\tilde G$ and a map of topological groups $\epsilon\colon \tilde G\to G$ that is also a weak equivalence.
\item There exists a classifying space $\tilde B$ for $G$ that is a countable CW-complex.
\item All homotopy groups of $G$ are countable.
\end{enumerate}
\end{prop}

\begin{proof}
We prove $(1)\Rightarrow(2)$. Suppose we are given a countable CW-group $\tilde G$ and a map of topological groups $\epsilon\colon \tilde G\to G$ that is a weak equivalence. By applying the functorial model of classifying space from Theorem \ref{classifyingspacefunctorial} we obtain a weak equivalence $B\tilde G\to BG$. Recall from Remark \ref{comparisonmodels} that the identity is a weak equivalence $\id\colon\tilde B\tilde G\to B\tilde G$. Therefore, $B\epsilon\colon\tilde B\tilde G\to B\tilde G\to BG$ is a weak equivalence. By pulling back $EG\to BG$ along this map we obtain a new universal $G$-bundle of which the base space is the countable CW-complex $\tilde B\tilde G.$

We prove $(2)\Rightarrow(1)$. Let $\tilde E$ be a universal $G$-bundle with a countable CW-complex
$\tilde B$ as base space. By Theorem \ref{classifyinggroup}, there exists a map of topological groups $\epsilon\colon \tilde\Omega\tilde B\to G$ such that $\epsilon_*(\tilde P\tilde B)\cong\tilde E$. The map induced by $\epsilon$
$$\tilde P\tilde B\cong\tilde P\tilde B\otimes_{\tilde\Omega\tilde B}\tilde\Omega\tilde B\to\tilde P\tilde B\otimes^\epsilon_{\tilde\Omega \tilde B}G=\epsilon_*(\tilde P\tilde B)\cong\tilde E$$
is a weak equivalence, as the source and the target are weakly contractible spaces, and is
compatible with the projection over $\tilde B$. It follows that the restriction to the fiber, which can be computed to be
$$\epsilon\colon \tilde\Omega\tilde B\cong \tilde\Omega\tilde B\otimes_{ \tilde\Omega\tilde B} \tilde\Omega\tilde B\to \tilde\Omega\tilde B\otimes^\epsilon_{\tilde\Omega\tilde B}G\cong G,$$
is therefore a weak equivalence, too.

We prove $(2)\Leftrightarrow(3)$. Let $B$ be a classifying space for $G$ that admits the structure of a CW-complex, which exists by Remark \ref{classifyingspaceCWcomplex}. Observe that the homotopy groups of $G$ are countable if and only if the homotopy groups of $B$ are countable. By \cite[Theorem 6.1, Page 137]{lw}, the homotopy groups of $B$ are countable if and only if $B$ is homotopy equivalent to a countable CW-complex $\tilde B$, which is necessarily a classifying space for $G$, too.
 \end{proof}
 
 The proposition motivates the following terminology.
 
\begin{dfnt}
If $G$ is a topological group satisfying the equivalent conditions of Proposition \ref{almostnicegroups}, we say that $G$ \textbf{admits a countable CW-replacement} or that it \textbf{admits a good classifying space}. Referring to the notation from Proposition \ref{almostnicegroups}, the group $\tilde G$ is a \textbf{countable CW-replacement} for $G$, the space $\tilde B$ is a \textbf{good classifying space} for $G$ and the group $\tilde\Omega X$ is a \textbf{good countable CW-replacement} for $G$. We denote
$$G^c:=\tilde\Omega\tilde B.$$
 \end{dfnt}
 
 The emphasis on $\tilde\Omega\tilde B$, as opposed to the more general $\tilde G$, is due to the fact that we have a good control on maps of topological groups defined on $\tilde\Omega\tilde B$, thanks to Theorem \ref{classifyinggroup}. Note also that all choices of $\tilde B$ lead to equivalent groups $G^c$.
 
 \vphantom{}
 
The notion of a topological group that admits a countable CW-replacement widely generalizes that of countable CW-group.

\begin{ex}
\begin{itemize}
\item Using \cite[Theorem 6.1, Page 137]{lw}, we obtain that any topological group whose underlying topological space is homotopy equivalent to a countable CW-complex has countable homotopy groups, and therefore admits a countable CW-replacement.
\item Using \cite[Corollary 1]{milnor3}, we obtain that any Lie group admits a countable CW-replacement, since the underlying manifold is separable.
\item For many interesting (orientable) manifolds $M$, the (orientation-preserving) diffeomorphism group $\textrm{Diff}(M)$ has countable homotopy groups and therefore admits a countable\linebreak CW-replacement. For instance, this is true when when $M=S^2$ (since $\textrm{Diff}(S^2)\simeq O(3)$ by the a theorem of Smale \cite{smale},
when $M$ is an orientable surface of genus $g\ge2$ such that $\pi_0(\textrm{Diff}(M))$ is countable (since by \cite[Theorem 1(c)]{ee} every connected component of  $\textrm{Diff}(M)$ is contractible),
when $M=S^3$ (since $\textrm{Diff}(S^3)\simeq O(4)$ by the Smale Conjecture proven by Hatcher in \cite{hatchersmaleconjecture}), when $M=S^1\times S^2$ (since $\textrm{Diff}(S^1\times S^2)\simeq O(2)\times O(3)\times\Omega O(3)$ by \cite{hatchersmaleconjecture2}), and when $M$ is a hyperbolic manifold of dimension 3 (as a consequence of \cite[Theorem 1.2]{gabai}).
\item Finite products of topological groups that admit a countable CW-replacement admit a countable CW-replacement.
\end{itemize}
\end{ex}

If a topological group $G$ admits a countable CW-replacement, the good replacement $G^c$ is enough to recover the theory of $G$-bundles.
Recall from Remark \ref{comparisonmodels} that for every countable CW-group $G$ the identity maps $\id\colon\tilde EG\to EG$ and $\id\colon\tilde BG\to BG$ are weak equivalences.

\begin{prop}
\label{cofibrantreplacementbundle}
Let $X$ be a countable CW-complexes and $G$ a topological group that admits a countable CW-replacement. A good CW-replacement $\epsilon\colon G^c\to G$ for $G$ induces a bijection
$${\epsilon}_*\colon{}_X{\mathcal Bun}_{G^{\text{c}}}/_{\cong}\to{}_X{\mathcal Bun}_{G}/_{\cong}.$$
\end{prop}

\begin{proof}
Using the fact that, by Proposition \ref{naturality} applied to the map $E\epsilon\colon EG^c\to EG$, there is an equivalence $(B\epsilon)^*(EG)\cong\epsilon_*(EG^c)$ of $G^c$-bundles over $BG$, one can check that the map $\epsilon_*$ fits into the following commutative diagram:
$$\xymatrix@R=.5cm@C=1cm{\mathcal Top(X,BG^c)/_\simeq\ar[r]^-{B\epsilon\circ-}\ar[d]_-{(-)^*(EG^c)}&\mathcal Top(X,BG)/_\simeq\ar[d]^-{(-)^*(EG)}\\
{}_X{\mathcal Bun}_{G^{\text{c}}}/_{\cong}\ar[r]_-{\epsilon_*}&{}_X{\mathcal Bun}_{G}/_{\cong}.}$$
We conclude that $\epsilon_*$ induces a bijection because the vertical maps are bijections by Theorem \ref{classificationGbundles}, and the top horizontal map is a weak equivalence because $B\epsilon$ is a weak equivalence and $X$ is a CW-complex.
\end{proof}

\begin{notation}
\label{notationc}
Given $G$ a topological group with a good CW-replacement $\epsilon\colon G^c\to G$, denote by $(-)^c$ the inverse of the map $\epsilon_*$ from Proposition \ref{cofibrantreplacementbundle}:
$$(-)^c\colon{}_X{\mathcal Bun}_{G}/_{\cong}\to{}_X{\mathcal Bun}_{G^{\text{c}}}/_{\cong}.$$
\end{notation}

\section{Characteristic classes as obstructions}
Given a topological group $G$ that admits a countable CW-replacement and a countable abelian group $Z$, the aim of this section is to construct a topological group $\hat G(c)$ and the homomorphism of topological groups $\gamma\colon G^c\to M^{k-1}Z$ associated to a universal characteristic class $c\in H^k(BG;Z)$, and prove that they both detect the vanishing of the characteristic class $c$ for a $G$-bundle.

\subsection{The looping homomorphism}
Milgram constructed in \cite{milgram} a model of classifying space for \emph{abelian} topological groups, that inherits the structure of a topological abelian group. We will need the following properties of Milgram's classifying space. Recall Milnor's model $\tilde B$ from Theorem \ref{niceclassifyingspace}.

\begin{teorema}
\label{classifyingspacemilgram}
Let $A$ be an abelian countable CW-group.
\begin{enumerate}
\item For any abelian countable CW-group $A$, there exists a functorial classifying space $MA$ that is an abelian countable CW-group.
\item There is a natural homotopy equivalence $\tau_A\colon \tilde BA\to MA$.
\end{enumerate}
\end{teorema}
\begin{proof} All the ingredients to prove (1) are from \cite{milgram}. The construction of the topological abelian group $MA$ follows from \cite[Corollary 1.7]{milgram}. In \cite[Theorem 2.3]{milgram}, a cell decomposition for $MA$ is constructed. When $A$ is a countable CW-complex, the cell decomposition constructed in \cite[Theorem 2.3]{milgram} recovers the topology, so $MA$ is a CW-complex. By reading through the definition of the cells, we see that the cells can be enumerated by finite collections of cells of $A$, and therefore there are countably many. Therefore $MA$ is a countable CW-complex. Moreover, it is proven in \cite[Section 2]{milgram} that the multiplication on $MA$ defined in \cite[Corollary 1.7]{milgram} is continuous and cellular. The fact that the inversion is cellular is straightforward, based on the explicit formula for inverses, in the proof of \cite[Corollary 1.7]{milgram}.
So $MA$ is an abelian countable CW-group.

For (2), one can read from \cite{milgram} and \cite{milnor2} the explicit constructions of $MA$ and $\tilde BA$, and see that they coincide respectively with the neat realization $|\textrm{Bar}(A)|$ and the fat realization $||\textrm{Bar}(A)||$ of the bar construction $\textrm{Bar}(A)$. We take $\tau_A$ to be the Bousfield-Kan map
	$$\tau_A\colon\tilde BA=||\textrm{Bar}(A)||\to|\textrm{Bar}(A)| = MA,$$
	which is natural.
	As the unit of $A$ is a vertex and $A^n$ is a CW-complex (because $A$ is a countable CW-complex), the simplicial space $\textrm{Bar}(A)$ is \emph{good} (according to \cite[Definition A.4]{segal}), and therefore $\tau_A$ is a weak equivalence by \cite[Proposition A.1(iv)]{segal}.
 \end{proof}

We are interested in iterating this construction.
\begin{notation}
Let $Z$ be a discrete countable abelian group and $k\ge0$. We denote by $M^kZ$ the topological abelian group given by
$$M^kZ:=\underbrace{M\circ\dots\circ M}_{k\text{ times}}(Z).$$
\end{notation}

It can be seen by induction on $n\ge0$ that $M^kZ$ has the homotopy type of an Eilenberg-Maclane space of type $(Z;k)$. Since $M^kZ$ is also an abelian topological group, by \cite[Theorem 1.2.4]{piccinini}, for any topological space $X$, the set of homotopy classes $[X,M^kZ]$ becomes an abelian group when endowed with pointwise multiplication.

\begin{prop}
\label{thirdiso}
Let $Z$ be a countable abelian group and $k\ge0$. For any CW-complex $X$ there is a natural isomorphism of abelian groups
$$H^k(X;Z)\cong [X,M^kZ].$$
\end{prop}

\begin{proof}
Recall from \cite{hatcher} that the ordinary cohomology $H^k(X;Z)$ can be identified with $[X,K(Z;k)]$ as a set, and with $[X,\Omega K(Z;k+1)]$ as a group, where the notation should emphasize that the group structure on $[X,\Omega K(Z;k+1)]$ is induced by the H-space structure of the loop space $\Omega K(Z;k+1)$.
Using \cite[Section 5]{milgram} there is a weak equivalence
$$M^{k}Z\stackrel{\simeq}{\longrightarrow}\Omega M^{k+1}Z$$
that is also maps of $H$-spaces, where $M^kZ$ is endowed with the abelian group structure and the H-space structure on $\Omega M^{k+1}Z$ is given by concatenation. This map is in fact a homotopy equivalence, because $M^kZ$ and $\Omega M^{k+1}Z$ are both CW-complexes using \cite[Corollary 5.3.7]{pf}.
Therefore, this map induces the desired isomorphism of abelian groups
$$[X,M^{k}Z]\cong[X,\Omega M^{k+1}Z]\cong[X,\Omega K(Z;k+1)]\cong H^k(X;Z).\qedhere$$
\end{proof}

Recall that for any topological group $G$ with classifying space $BG$ the elements of $H^k(BG,Z)$ are called \textbf{universal characteristic classes}. We refer the reader to \cite{ms} for more details on this account.

\begin{construction}
\label{cocycle}
Let $G$ be a topological group that admits a countable CW-replacement, $\tilde B$ a classifying space for $G$ that has the homotopy type of a countable CW-complex, $Z$ a discrete countable abelian group and $k>0$. By Proposition \ref{classifyingspacemilgram}(2) and Theorem \ref{classificationnice}, the cohomology of $\tilde B$ can be identified with a quotient of $\mathcal Gp(\mathcal Top)(G^c,M^{k-1}Z)$ as follows:
$$H^k(\tilde B;Z)\cong\mathcal Top(\tilde B,M^kZ)/_{\simeq}\cong\mathcal Top(\tilde B,\tilde BM^{k-1}Z)/_{\simeq}=\mathcal Gp(\mathcal Top)(G^c,M^{k-1}Z)/_{\equiv}.$$
Note that a different choice for $\tilde B$ would lead to a compatible identification.
In particular, every characteristic class $c\in H^k(\tilde B;Z)$ in positive degree $k>0$ is represented by a homomorphism of topological groups
$$\gamma\colon G^c\to M^{k-1}Z,$$
which is determined up to algebraic equivalence. Note that $\gamma$ determines the cohomology class $c$ completely, as $c$ is represented by the homotopy class of
$$\xymatrix{c\colon\tilde B\simeq\tilde BG^c\ar[r]^-{\tilde B\gamma}&\tilde BM^{k-1}Z\simeq K(Z;n)}.$$
We therefore refer to $\gamma$ as the \textbf{looping homomorphism} of $c$.
\end{construction}

The looping homomorphism of of the first Stiefel-Whitney map is essentially the determinant map.

\begin{ex}
\label{examplesloopinghomomorphism}
Let $k=1$, $Z=\mathbb Z/2$, $G=O(n)$
and $c:=w_1\in H^1(BO(n);\mathbb Z/2)$ the first universal Stiefel-Whitney class. We construct (a representative for) the looping homomorphism $\omega_1\colon O(n)^c\to M^0(\mathbb Z/2)=\mathbb Z/2$ associated to $w_1$ as follows. Notice that the real Stiefel manifold $V(\mathbb R^\infty;n)$ is an $O(n)$-universal bundle with corresponding classifying space the real Grassmannian $Gr(\mathbb R^\infty;n)$, that admits the structure of a countable CW-complex. Use Theorem \ref{classificationnice} to assert the existence of a homomorphism of topological groups $O(n)^c=\tilde\Omega Gr(\mathbb R^\infty;n)\to O(n)$ such that
$$(\tilde\Omega Gr(\mathbb R^\infty;n)\to O(n))_*\tilde PGr(\mathbb R^\infty,n)\cong V(\mathbb R^\infty,n)\cong(\id\colon Gr(\mathbb R^\infty,n)\to Gr(\mathbb R^\infty,n))^*V(\mathbb R^\infty,n).$$
Such a map $O(n)^c\to O(n)$ is necessarily a weak equivalence by Remark \ref{adjunctionweakequivalence}. 
Then the homomorphism $\omega_1$ can be taken to be the composite
$$\omega_1\colon\colon \xymatrix{O(n)^c\ar[r]^-\sim&O(n)\ar[r]^-{\det}&O(1)\cong\mathbb Z/2=M^0(\mathbb Z/2)}.$$
\end{ex}

\subsection{The reduction group}
The following proposition exhibits a convenient model for the homotopy fiber of $\tilde B\gamma\colon\tilde BG^c\to\tilde BM^{k-1}Z$. Recall the notation for the tensor product $\alpha_* E =E\otimes^{\alpha}_GA$.

\begin{prop}
\label{fibersequence1}
Let $\alpha\colon G\to A$ be a homomorphisms of countable CW-groups.
\begin{enumerate}
\item The space $\alpha_* \tilde EG$ is a countable CW-complex.
\item The space $\alpha_* \tilde EG$ is the homotopy fiber of $\tilde B\alpha\colon\tilde BG\to\tilde BA$.
\item The space $\alpha_* \tilde EG$ is connected if and only if $\alpha$ is $\pi_0$-surjective.
\end{enumerate}
\end{prop}

\begin{proof}[Proof of Proposition \ref{fibersequence1}]
The space $\alpha_* \tilde EG =\tilde EG\otimes_GA$ is an $A$-bundle over $\tilde BG$.
By applying Proposition \ref{totalspaceCWcomplex} to $\alpha_* \tilde EG$, we conclude that $\alpha_* \tilde EG$ is a countable CW-complex, which proves (1).

The argument for (2) is a variant of the one used in \cite[Proof of Theorem 10.1]{mitchell2011}. The tensor product $ \tilde EG\otimes_G \tilde EA$, obtained by coequalizing the actions of $G$ on $\tilde EA$ and $\tilde EG$, is a fiber bundle over the countable CW-complex $\tilde BA$ with fiber the space $\tilde EG\otimes_GA$, which admits the structure of a countable CW-complex by (1).
Again by Proposition \ref{totalspaceCWcomplex}, $ \tilde EG\otimes_G \tilde EA$ admits therefore the structure of a countable CW-complex. We then consider the map
$$( \tilde EG,\tilde E\alpha)\colon  \tilde EG\to \tilde EG\times  \tilde EA,$$
a morphism of right $G$-spaces (where the action on $\tilde EG\times \tilde EA$ is on both factors), which are in fact $G$-bundles. The morphism $( \tilde EG,\tilde E\alpha)$ induces therefore a map at the level of quotients:
$$\overline{( \tilde EG,\tilde E\alpha)}\colon\tilde BG\to \tilde EG\otimes_G \tilde EA.$$
Moreover $( \tilde EG,\tilde E\alpha)$ is a weak homotopy equivalence (since its source and target are contractible), and restricts to a homeomorphism on the distinguished fibers because of the equivariance condition. By naturality of the homotopy long exact sequence induced by a fibration, the induced map $\overline{( \tilde EG,\tilde E\alpha)}$ is also a weak equivalence.
As $ \tilde BG$ and $ \tilde EG\otimes_G \tilde EA$ are (countable) CW-complexes, this map is in fact a homotopy equivalence.
Since the following diagram commutes
$$\xymatrix@R=.5cm{ \tilde EG\otimes_G \tilde EA\ar@{->>}[rr]^-{\pi'\circ\pr_2}&& \tilde BA\ar@{=}[d]\\
\ar[u]_-{\simeq}^-{\overline{(\tilde EG,\tilde E\alpha)}} \tilde BG\ar[rr]_-{\tilde B\alpha}&& \tilde BA,
}$$
the homotopy fiber of $\tilde B\alpha$ is equivalent to the (homotopy) fiber of the bundle projection\linebreak$\tilde EG\otimes_G \tilde EA\to\tilde BA$, which is $\tilde EG\otimes_GA= \alpha_* \tilde EG$, as desired.

For (3), we observe that, using the homotopy long exact sequence induced by fibrations, there is an exact sequence (of pointed sets)
$$\xymatrix{\pi_0(G)\ar[r]^-{\pi_0(\alpha)}&\pi_0(A)\ar[r]&\pi_0(\alpha_*\tilde EG )\ar[r]&0.}\qedhere$$
\end{proof}

We can now define the reduction group associated to a universal characteristic class.

\begin{dfnt}
\label{reductiongroup}
Let $\alpha\colon G\to A$ be a $\pi_0$-surjective homomorphisms of countable CW-groups.
The \textbf{reduction group} of $\alpha$ is the countable CW-group
$$\hat G(\alpha):=\tilde\Omega(\alpha_*\tilde EG ).$$
It comes with a map of topological groups $j\colon\hat G(\alpha)\to G$,
which is determined up to algebraic equivalence, obtained by adjoining the projection $\alpha_\star(\tilde EG)\to\tilde BG$ according to Theorem \ref{classificationnice}.
Note that the condition of $\pi_0$-surjectivity for $\alpha$ is automatically satisfied when $A$ is connected.
\end{dfnt}

We now focus on reduction groups produced by looping homomorphisms associated to universal characteristic classes.
\begin{dfnt}
\label{reductiongroupcharacteristicclass}
Let $G$ be a topological group that admits a countable CW-replacement, $Z$ a countable abelian group and $k>0$. Let $c\in H^k(BG;Z)$ be a universal characteristic class represented by a $\pi_0$-surjective looping homomorphism $\gamma\colon G^c\to M^{k-1}Z$ (as in Construction \ref{cocycle}). The \textbf{reduction group} of $c$ is the countable CW-group
$$\hat G(c):=\hat G(\gamma)=\tilde\Omega(\gamma_*\tilde EG^c ).$$
It comes with a homomorphism $\hat G(c)\stackrel{j}{\rightarrow}G^c\stackrel{\epsilon}{\rightarrow}G$, where $\epsilon$ is the countable CW-replacement, and the map $j$ is from Definition \ref{reductiongroup}.
Note that $\gamma:G^c\to M^{k-1}Z$ is $\pi_0$-surjective if and only if the cohomology class $c:\tilde BG^c\to K(Z;k)$ is $\pi_1$-surjective, which is automatically satisfied when $k>1$.
\end{dfnt}

The reduction group of a map $\alpha$ has a sort of universal property.
\begin{prop}
\label{universalpropertyreductiongroup}
Let $\alpha\colon G\to A$ be a $\pi_0$-surjective homomorphisms of countable CW-groups.
\begin{enumerate}
\item The classifying space of $\hat G(\alpha)$ is the homotopy fiber of $\tilde B\alpha\colon\tilde BG\to\tilde BA$.
\item If $G'$ is another topological group whose classifying space is the homotopy fiber of $\tilde B\alpha\colon\tilde BG\to\tilde BA$, there exists a homomorphism of topological groups $\hat G(\alpha)\to G'$ that is also a weak equivalence.
\end{enumerate}
\end{prop}

\begin{proof}
For (1), we note that $\tilde P(\alpha_*\tilde EG)$ is by Theorem \ref{classifyinggroup} a universal $\tilde\Omega(\alpha_*\tilde EG)$-bundle, i.e., a universal $\hat G(\alpha)$-bundle, with base space $\alpha_*\tilde EG$. Therefore, $\alpha_*\tilde EG$ is a classifying space for $\hat G(\alpha)$, and is therefore equivalent to $\tilde B\hat G(\alpha)$, which was proven in Proposition \ref{fibersequence1} to be the homotopy fiber of $\tilde B\alpha$.

For (2), we suppose that $G'$ is a topological group with a classifying space $B'$ that is the homotopy fiber of $\tilde B\alpha$. Without loss of generality, we can assume that $B'$ is a CW-complex, and is therefore homotopy equivalent to $\alpha_*\tilde EG$. Let $h\colon \alpha_*\tilde EG\to B'$ be a homotopy equivalence. Let $E'$ be a universal $G'$-bundle whose base space is $B'$.  Then $h^*E'$ is a universal $G'$-bundle, with base space $\alpha_*\tilde EG$. By Theorem \ref{classifyinggroup}, there exists a homomorphism of topological groups $a\colon \hat G(\alpha)=\tilde\Omega(\alpha_*\tilde EG)\to G'$ such that $h^*E'\cong a_*\tilde P(\alpha_*\tilde EG)$. Using (a variant of) Remark \ref{adjunctionweakequivalence}, we conclude that $a$ is a weak equivalence, as desired.
\end{proof}

This proposition can be used to check that the general construction $\hat G(c)$ recovers certain important topological groups. In particular, for certain interesting characteristic classes $c$, the reduction group $\hat G(c)$ is equivalent to a Lie group.
\begin{ex}
\label{standardmodels}
\begin{itemize}
	\item[>>] When $k=1$, $Z=\mathbb Z/2$, $G=O(n)$ and $c=w_1$ is the first Stiefel-Whitney class, it is known that the classifying space of the special orthogonal group $BSO(n)$ of $SO(n)$ fits into a homotopy fiber sequence
$$\xymatrix{BSO(n)\ar[r]&BO(n)\ar[r]^-{w_1}&K(\mathbb Z/2;1).}$$
By Proposition \ref{universalpropertyreductiongroup}(2), there exists a homomorphism of topological groups $\hat G(w_1)\to SO(n)$ which is also a weak equivalence.
	\item[>>] When $k=2$, $Z=\mathbb Z/2$, $G=SO(n)$ and $c=w_2$ is the second Stiefel-Whitney class, it is known that the classifying space $BS\text{pin}(n)$ of the spin group $S\text{pin}(n)$ fits into a homotopy fiber sequence
	$$\xymatrix{BS\text{pin}(n)\ar[r]&BSO(n)\ar[r]^-{w_2}&K(\mathbb Z/2;2).}$$
	By Proposition \ref{universalpropertyreductiongroup}(2), there exists a homomorphism of topological groups $\hat G(w_2)\to Spin(n)$ which is also a weak equivalence.
	\item[>>] When $k=4$, $Z=\mathbb Z/2$, $G=Spin(n)$ and $c=\frac12p_1$ is the first fractional Pontryagin class, it is known that the classifying space $BString(n)$ of the string group $String(n)$ fits into a homotopy fiber sequence
$$\xymatrix{BString(n)\ar[r]&BSpin(n)\ar[r]^-{\frac12p_1}&K(\mathbb Z/2;4).}$$
By Proposition \ref{universalpropertyreductiongroup}(2), there exists a homomorphism of topological groups $\hat G(\frac12p_1)\to String(n)$ which is also a weak equivalence.	
	\item[>>] When $k=2$, $Z=\mathbb Z$, $G=U(n)$ and $c=c_1$ is the first Chern class, it is known that the classifying space of the special unitary group $BSU(n)$ of $SU(n)$ fits into a homotopy fiber sequence
$$\xymatrix{BSU(n)\ar[r]&BU(n)\ar[r]^-{c_1}&K(\mathbb Z;2).}$$
By Proposition \ref{universalpropertyreductiongroup}(2), there exists a homomorphism of topological groups $\hat G(c_1)\to SU(n)$ which is also a weak equivalence.	
	\item[>>] When $k=2n$, $Z=\mathbb Z$, $G=U(n)$ and $c=c_n$ is the top Chern class, by Proposition \ref{universalpropertyreductiongroup}(1) $\tilde B\hat G(c_n)$ is the homotopy fiber of $c_n$.
	Using the standard identification of the cohomology of complex grassmannians with polynomial algebras, one can prove that the composite $\tilde BU(n-1)\to\tilde BU(n)\stackrel{c_n}{\rightarrow} K(\mathbb Z;2n)$ represents a trivial cohomology class in $H^{2n}(\tilde BU(n-1),\mathbb Z)$. Therefore $\tilde BU(n-1)\to\tilde BU(n)$ lifts up to homotopy to $\tilde B\hat G(c_n)$,
$$\xymatrix@R=.7cm{\tilde B\hat G(c_n)\ar[r]\ar@{}[rd]|(.3){\simeq}&BU(n)\ar@{}[d]|(.5){\simeq}\ar[r]^-{c_n}&K(\mathbb Z;2n)\\
\tilde BU(n-1)\ar[ru]\ar@/_1.5pc/[rru]_-0\ar@{-->}[u]&&}$$
By adjoining (according to Theorem \ref{classificationnice}) the homotopy commutative triangle to the left, with a little work one can build a diagram of topological groups of the following form, which commutes up to algebraic equivalence:
$$\xymatrix@R=.7cm{\hat G(c_n)^c\ar[r]\ar@{}[rd]|(.3){\equiv}|(.7){\equiv}&U(n).\\
U(n-1)^c\ar[ru]\ar[u]\ar[r]_-\sim&U(n-1)\ar@{^{(}->}[u]}$$
This can be rephrased by saying that, modulo replacing $U(n-1)$ with $U(n-1)^c$, the standard inclusion $U(n-1)\hookrightarrow U(n)$ factors through the reduction group $\hat G(c_n)$ of the top Chern class.

\end{itemize}
\end{ex}

The following remark gives a fair description of the homotopy groups of reduction groups associated to universal characteristic classes: the homotopy groups of $\hat G(\gamma)$ coincide with the homotopy groups of $G$ except possibly in degree $k-1$ and $k-2$.

\begin{rem}
\label{homotopygroupsreductiongroups}
Let $Z$ be a countable abelian group, $G$ a countable CW-group and $k\ge1$. Let $\alpha\colon G\to M^{k-1}Z$ be a $\pi_0$-surjective homomorphisms. We know by Proposition \ref{fibersequence1} that $\tilde B\hat G(\alpha)$ fits into a homotopy fiber sequence
$$\xymatrix{\tilde B\hat G(\gamma)\ar[r]&\tilde BG\ar[r]^-{\tilde B\alpha}&\tilde B M^{k-1}Z,}$$
and that $\tilde B M^{k-1}Z$ is an Eilenberg-MacLane space of type $K(Z;k)$. This implies that
$$\pi_i(\hat G(\gamma))\cong\pi_i(G)\text{ for }i\neq k-1,k-2,$$
and the remaining homotopy groups fit into an exact sequence of the form
$$\xymatrix{0\ar[r]&\pi_{k-1}(\hat G(\gamma))\ar[r]^-{}&\pi_{k-1}(G)\ar[r]&Z\ar[r]&\pi_{k-2}(\hat G(\gamma))\ar[r]^-{}&\pi_{k-2}(G)\ar[r]&0}.$$
\end{rem}

This fact can be used to build a \emph{Whitehead tower} \cite{whiteheadtower} of any topological group $G$ whose homotopy groups are countable (e.g. Lie groups). Since every layer of the Whitehead tower shares similarities with the \emph{cellularization} of \cite{farjoun}, inspired by \cite[Chapter 3, Theorem A.2]{farjoun} one could wonder whether every layer of the Whitehead tower of an H-space is itself an H-space. The following construction shows that the topological group $G$ admits a Whitehead tower made by topological groups (and homomorphisms of such).

\begin{construction}
\label{whiteheadtower}
Let $G$ be a topological group that admits a countable CW-replacement $\tilde G$. We inductively define a countable CW-group $G^{(n)}$ together with maps
$$j_n\colon G^{(n)}\to G^{(n-1)}\to\dots\to G^{(0)}\to G$$
that induce isomorphisms
$$\pi_i(G^{(n)})\cong\pi_i(G^{(n-1)})\cong\dots\cong\pi_i(G^{(0)})\cong\pi_i(G)\text{ for }i\ge n.$$
and such that
$$\pi_i(G^{(n)})\cong0\text{ for }i<n.$$
We first set $G^{(0)}:=\tilde G$, together with the map $\epsilon\colon \tilde G\to G$ that gives the countable CW-replacement. Suppose then $n>0$ and that $G^{(n-1)}$ and the map $j_{n-1}\colon G^{(n-1)}\to G$ have been constructed. By the universal coefficient theorem for cohomology and the Hurewicz Theorem, there is an isomorphism of abelian groups
$$\begin{array}{rclr}
\left[\tilde B(G^{(n-1)}),K\left(\pi_{n}(BG),n\right)\right]&\cong&H^{n}\left(\tilde B(G^{(n-1)});\pi_{n}(BG)\right)&\\
&\cong&\textrm{Hom}\left(H_{n}(\tilde B(G^{(n-1)});\mathbb Z),\pi_{n}(BG))\right)&\\
&\cong&\textrm{Hom}\left(\pi_{n}(\tilde B(G^{(n-1)})),\pi_{n}(BG))\right)&\\
&\cong&\textrm{Hom}\left(\pi_{n-1}(G^{(n-1)}),\pi_{n-1}(G))\right).&\\
\end{array}$$
One can check that the isomorphism above is precisely induced by applying $\pi_{n-1}$.
One of the elements of $\textrm{Hom}\left(\pi_{n-1}(G^{(n-1)}),\pi_{n-1}(G))\right)$ is the isomorphism $\pi_n(j_{n-1})\colon \pi_{n-1}(G^{(n-1)})\cong\pi_{n-1}(G)$
induced by the map $j_{n-1}\colon G^{(n-1)}\to G$. Let $c_n\colon\tilde B(G^{(n-1)})\to K(\pi_n(BG),n)$ be the element in $H^n(\tilde B(G^{(n-1)});\pi_n(BG))$ corresponding to $\pi_n(j_{n-1})$.
We then set $G^{(n)}:=\hat G(c_n)$,
together with the map $j_n\colon G^{(n)}=\hat G(c_n)\to G^{(n-1)}\stackrel{j_{n-1}}{\longrightarrow} G$.
\end{construction}

\begin{teorema}
Let $G$ be a topological group that admits a countable CW-replacement.
The tower of groups
$$\dots\to G^{(n+1)}\to G^{(n)}\to G^{(n-1)}\to\dots\to G^{(0)}\to G$$
from  Construction \ref{whiteheadtower} is a Whitehead tower for $G$.
\end{teorema}

\begin{proof}
The map $j_0:=\epsilon\colon \tilde G=G^{(0)}\to G$ is a weak equivalence by definition.
By construction, the exact sequence Remark \ref{homotopygroupsreductiongroups} (with $k=n$, $Z=\pi_n(BG)\cong\pi_{n-1}(G)$ and $\gamma=c_n$) becomes of the following form:
$$\xymatrix@C=.54cm{0\ar[r]&\pi_{n-1}(G^{(n)})\ar[r]^-{}&\pi_{n-1}(G^{(n-1)})\ar[rr]^-\cong_-{\pi_{n-1}(j_{n-1})}&&\pi_{n-1}(G)\ar[r]&\pi_{n-2}(G^{(n)}))\ar[r]^-{}&\pi_{n-2}(G^{(n-1)})\ar[r]&0},$$
which implies the vanishing of the homotopy groups of $G^{(n)}$ in degree $i=n-1,n-2$:
$$\pi_{n-2}(G^{(n)})\cong\pi_{n-2}(G^{(n-1)})\cong0\quad\text{ and }\quad\pi_{n-1}(G^{(n)})\cong0.$$
On the other hand, by Remark \ref{homotopygroupsreductiongroups}, we know that the homotopy groups in degree $i\neq n-2,n-1$ are given by
$$\pi_i(G^{(n)})=\pi_i(\hat G(c_n))\cong\pi_i(G^{(n-1)})=
\left\{
\begin{array}{ll}
0&\text{ if }i<n-2\\
\pi_i(G^{(n-1)})\cong\dots\cong\pi_i(G^{(0)})\cong\pi_i(G).&\text{ if }i>n-1
\end{array}
\right.$$
We therefore conclude that $\left(G^{(n)}\right)_{n\ge0}$ is a Whitehead tower for $G$.
\end{proof}

\subsection{Characteristic classes as obstructions to group reduction}

Given $G$ a topological group whose homotopy groups are countable (e.g., a Lie group) and a universal characteristic class $c$, we constructed its reduction group $\hat G(c)=\hat G(\gamma)$, together with a map $j:\hat G(\gamma)\to G$. In this section we prove that it detects the vanishing of the characteristic class $c$ for a $G$-bundle.
The first step is the following proposition.
\begin{prop}
\label{groupreduction}
Let $\alpha\colon G\to A$ be a $\pi_0$-surjective homomorphisms of countable CW-groups. For $E$ a $G$-bundle over a CW-complex $X$ classified by a map $t\colon X\to\tilde BG$, the following are equivalent.
\begin{itemize}
	\item[(a)] There exists a $\hat G(\alpha)$-bundle $\hat E$ over $X$ such that
$\hat E\otimes_{\hat G(\alpha)}G=j_*\hat E\cong E$.
	\item[(b)] The map $\tilde B\alpha\circ t\colon X\to\tilde BA$ is nullhomotopic.
\end{itemize}
\end{prop}
Proving Theorem \ref{groupreduction} from Proposition \ref{fibersequence1} is a standard argument, that we briefly recall.
\begin{proof}
Condition (b) says that $t$ is killed by $\tilde B\alpha$, which is equivalent to saying that $t$ factors up to homotopy through the homotopy fiber $\tilde B\hat G(\alpha)$ via a map $\hat t\colon X\to\tilde B\hat G(\alpha)$, as displayed
$$\xymatrix@R=.3cm{
&X\ar[d]^-t\ar[rd]^-0\ar@{-->}[ld]_-{\hat t}&\\
\tilde B\hat G(\alpha)\ar[r]^-{}_-{\tilde Bj}&\tilde BG\ar@{}[ru]|(.3){\simeq}\ar@{}[lu]|(.3){\simeq}\ar[r]_-{\tilde B\alpha}&\tilde BA.}$$
Note that, by Proposition \ref{naturality}, the map $\tilde Ej\colon\tilde E\hat G(\gamma)\to\tilde EG$ induces an equivalence
$$j_*\tilde E\hat G(\alpha)\cong(\tilde Bj)^*\tilde EG.$$
Finally, if we then let $\hat E:=\hat t^*\tilde E\hat G(\alpha)$, we have that
$$j_*\hat E\cong j_*\hat t^* \tilde E\hat G(\alpha)\cong\hat t^* j_* \tilde E\hat G(\alpha)\cong\hat t_* (Bj)^* \tilde EG\cong(\tilde Bj\circ\hat t)^* \tilde EG.$$
So $t$ is homotopic to $Bj\circ\hat t$ if and only if $j_* \hat E\cong t^* EG\cong E$, namely Condition (a).
\end{proof}

We are interested in applying the previous proposition to a homomorphism $\gamma\colon G^c\to M^{k-1}Z$ that represents a universal characteristic class $c$, as in Construction \ref{cocycle}. The following result shows that the universal characteristic class $c$ measures the obstruction to the reduction to a $\hat G(c)$-bundle. Recall the map $\epsilon\circ j\colon \hat G(c)\to G^c\to G$ from Definition \ref{reductiongroupcharacteristicclass}.
\begin{teorema}
\label{groupreductioncor}
Let $G$ be a topological group that admits a countable CW-replacement, $Z$ a countable abelian group and $k>0$. Let $c\in H^k(BG;Z)$ be a $\pi_1$-surjective universal characteristic class. For $E$ a $G$-bundle over a CW-complex $X$, the following are equivalent.
\begin{itemize}
	\item[(A)] There exists a $\hat G(c)$-bundle $\hat E$ over $X$ such that $\hat E\otimes_{\hat G(\gamma)}G=(\epsilon\circ j)_*\hat E\cong E$.
		\item[(B)] The characteristic class $c$ of $E$ vanishes, i.e., $c(E)=0\in H^k(X;Z)$.
\end{itemize}
\end{teorema}

\begin{proof}
By applying Proposition \ref{groupreduction} to $E^{\text{c}}$ (classified by a map $t\colon X\to\tilde BG^c$ with respect to the universal bundle $\tilde EG^c$) and to the looping homomorphism
$\gamma\colon G^{\text{c}}\to M^{k-1}Z$ associate to $c$ as in Construction \ref{cocycle}, we obtain that the following are equivalent.
\begin{itemize}
	\item[(a)] There exists a $\hat G(c)$-bundle $\hat E$ over $X$ and an equivalence $j_*\hat E\cong E^{\text{c}}$ of $G^{\text{c}}$-bundles over $X$.
	\item[(b)] The map $\tilde B\gamma\circ t\colon X\to\tilde BG^c\to\tilde BM^{k-1}Z$ is nullhomotopic.
\end{itemize}

The conditions (A) and (a) are equivalent by Proposition \ref{cofibrantreplacementbundle}. Indeed condition (a) is obtained from condition (A) by applying $(-)^c$, and (A) is obtained from (a) by applying ${\epsilon}_*$.

The conditions (b) and (B) are equivalent because the cohomology class $c(E)$ is represented by $\tilde B\gamma\circ t:X\to\tilde BM^{k-1}Z$. Indeed,
the same $t\colon X\to \tilde BG^c$ is also a classifying map for $E$ (with respect to the universal $G$-bundle $(B\epsilon\colon\tilde BG^c\to BG^c\to BG)^*(EG)$, and $\tilde B\gamma$ represents the cohomology class of the universal characteristic class $c$ (by Construction \ref{cocycle}).
\end{proof}

We use this theorem to recover and strengthen certain classical statements (cf. \cite{ms}). Recall that certain matrix groups $G'$ have been identified in Example \ref{standardmodels} with the reduction groups $\hat G(c)$ associated to certain characteristic classes $c$, via a map $\hat G(c)\to G'$ which is also a weak equivalence. By Proposition \ref{cofibrantreplacementbundle}, such a map induces a bijection
between the set of equivalence classes of $G'$-bundles over $X$ and the set of equivalence classes of $\hat G(c)$-bundles over $X$.
\begin{ex}
\begin{itemize}
	\item[>>] As an application of Theorem \ref{groupreductioncor}, the first Stiefel-Whitney class $w_1(E)$ of an $O(n)$-bundle $E$ vanishes if and only if there exists an $SO(n)$-bundle $\hat E$ such that $E\cong\hat E\otimes_{SO(n)}O(n)$, which coincides with one of the standard notions of orientability.
	\item[>>] As an application of Theorem \ref{groupreductioncor}, the second Stiefel-Whitney class $w_2(E)$ of an $SO(n)$-bundle $E$ vanishes if and only if there exists a $Spin(n)$-bundle $\hat E$ such that $E\cong\hat E\otimes_{Spin(n)}SO(n)$, which is equivalent the usual notion of existence of a spin structure.
	\item[>>] As an application of  Theorem \ref{groupreductioncor}, the first fractional Pontryagin class $\frac12p_1(E)$ of a $Spin(n)$-bundle $E$ vanishes if and only if there exists a $String(n)$-bundle $\hat E$ such that $E\cong\hat E\otimes_{String(n)}Spin(n)$, which refines (but is equivalent to) the usual notion of existence of a string structure.
	\item[>>] As an application of  Theorem \ref{groupreductioncor}, the first Chern class $c_1(E)$ of a $U(n)$-bundle $E$ vanishes if and only if there exists an $SU(n)$-bundle $\hat E$ such that $E\cong\hat E\otimes_{SU(n)}U(n)$.
	\item[>>] As an application of  Theorem \ref{groupreductioncor}, the top Chern class $c_n(E)$ of a $U(n)$-bundle $E$ vanishes if and only if there exists a $\hat G(c_n)$-bundle $\hat E$ such that $E\cong\hat E\otimes_{\hat G(c_n)}U(n)$.
 \item[>>] If $E$ is the frame bundle of a complex vector bundle $\mathcal E$ of rank $n$, we observe that $\mathcal E$ admits a nowhere vanishing section if and only if there exists a $U(n-1)$-bundle $\hat E'$ such that $E\cong\hat E'\otimes_{U(n-1)}U(n)$ (see \cite[Theorem 2.27]{cohen}). In this case, by setting $\hat E:=\hat E'\otimes_{U(n-1)}\hat G(c_n)$ and recalling from Example \ref{standardmodels} that the inclusion of $U(n-1)\hookrightarrow U(n)$ essentially factors through $\hat G(c_n)$, we obtain that
 $$E\cong \hat E'\otimes_{U(n-1)}U(n)\cong\hat E'\otimes_{U(n-1)}\hat G(c_n)\otimes_{\hat G(c_n)}U(n)=\hat E\otimes_{\hat G(c_n)}U(n).$$
In this case, Theorem \ref{groupreductioncor} tells us that the top Chern class $c_n(E)=c_n(\mathcal E)$ vanishes.
 This argument (re)proves that well-known fact that the top Chern class (i.e., the Euler class) of a complex vector bundle is an obstruction to the existence of a nowhere vanishing section. The fact that it is not the only obstruction is a symptom of the fact that $\hat G(c_n)$ is not homotopy equivalent to $U(n-1)$.
\end{itemize}
\end{ex}

\subsection{Characteristic classes as obstructions to the extension of the looping homomorphism}
Given $G$ a topological group whose homotopy groups are countable (e.g., a Lie group) and a universal characteristic class $c$, we constructed its looping homomorphism $\gamma$. In this section we prove that $\gamma$ detects the vanishing of the characteristic class $c$ for a $G$-bundle. The first step is the following proposition.

\begin{prop}
\label{retraction}
Let $\alpha\colon G\to A$ be a $\pi_0$-surjective homomorphisms of countable CW-groups. For $E$ a $G$-bundle over a CW-complex $X$ classified by a map $t\colon X\to\tilde BG$, the following are equivalent.
\begin{itemize}
	\item[(b)] The map $\tilde B\alpha\circ t\colon X\to\tilde BA$ is nullhomotopic.
	\item[(c)] The map of topological groups $\alpha\colon G\to A$ extends (along any $G$-equivariant inclusion $G\to E$) to an $\alpha$-equivariant map $E\to A$.
	\item[(d)] The bundle $\alpha_*\tilde EG$ is trivial, i.e, $\alpha_* E\cong X\times A$.
\end{itemize}
\end{prop}
\begin{proof}
We prove $[(b)\Leftrightarrow(d)]$. Using Proposition \ref{naturality} on the map $\tilde Ea\colon \tilde EG\to\tilde EA$, we obtain equivalences of $A$-bundles over $X$:
$$(\tilde Ba\circ t)^*\tilde EA\cong t^*(\tilde Ba)^*\tilde EA\cong t^* a_*\tilde EG\cong a_* t^*\tilde EG\cong a_* E.$$
By Theorem \ref{classificationnice}, the map $\tilde Ba\circ t$ is nullhomotopic if and only if $a_* E$ is equivalent to the trivial bundle.

We prove $[(c)\Rightarrow(d)]$. If there exists an $\alpha$-equivariant map $\kappa\colon E\to A$, then
$$(\pi,\kappa)\colon E\to X\times A$$
is an $\alpha$-equivariant map over $X$, which induces by Proposition \ref{naturality} an equivalence of $A$-bundles
$$\overline{(\pi,\kappa)}\colon \alpha_* E\cong X\times A.$$

We prove $[(d)\Rightarrow(c)]$. Given an equivalence $\phi\colon \alpha_* E\cong X\times A$ of $A$-bundles over $X$, we use it to construct the $\alpha$-equivariant map
$$\xymatrix{E\cong E\otimes_G G\ar[r]^-{E\otimes_G\alpha}&E\otimes_{G}A=\alpha_* E\ar[r]^-{\phi}_-\cong&X\times A\ar[r]^-{\pr_2}&A,}$$
which can be checked to extend $\alpha$.
\end{proof}

Again, we are interested in applying the previous proposition to the looping homomorphism $\gamma\colon G^c\to M^{k-1}Z$ associated to a universal characteristic class $c$, as in Construction \ref{cocycle}. The following result shows that the $\pi_1$-surjective characteristic class $c$ measures the obstruction to the equivariant extension of $\gamma$ to the whole bundle $E^c$.

\begin{teorema}
\label{retractioncor}
Let $G$ be a topological group that admits a countable CW-replacement, $Z$ a countable abelian group and $k>0$. Let $c\in H^k(BG;Z)$ be a $\pi_1$-surjective universal characteristic class with looping homomorphism $\gamma\colon G^{\textnormal{c}}\to M^{k-1}Z$. For $E$ a $G$-bundle over a CW-complex $X$, the following are equivalent.
\begin{itemize}
	\item[(B)] The characteristic class $c(E)$ of $E$ vanishes, i.e., $c(E)=0\in H^k(X;Z)$.
	\item[(C)] The map $\gamma\colon G^{\textnormal{c}}\to M^{k-1}Z$ extends to a $\gamma$-equivariant map $E^{\textnormal{c}}\to M^{k-1}Z$.
\end{itemize}
\end{teorema}

\begin{proof}
By applying Proposition \ref{retraction} to $E^{\text{c}}$ (classified by a map $t\colon X\to\tilde BG^c$ with respect to the universal bundle $\tilde EG^c$) and to the looping homomorphism
$\gamma\colon G^{\text{c}}\to M^{k-1}Z$ associate to $c$ as in Construction \ref{cocycle}, we obtain that the following are equivalent.
\begin{itemize}
	\item[(b)] The map $\tilde B\gamma\circ t\colon X\to\tilde BG^c\to\tilde BM^{k-1}Z$ is nullhomotopic.
	\item[(C)] The map $\gamma\colon G^{\textnormal{c}}\to M^{k-1}Z$ extends to a $\gamma$-equivariant pointed map $E^{\textnormal{c}}\to M^{k-1}Z$.
\end{itemize}
The equivalence between (b) and (B) was already established in the proof of Theorem \ref{groupreductioncor}.
\end{proof}

The theorem recovers another characterization of orientability.

\begin{ex}
We apply the theorem, and use the fact that the looping homomorphism of the fist Stiefel-Whitney class has been identified with the determinant map in Example \ref{examplesloopinghomomorphism}.
We obtain that the first Stiefel-Whitney class $w_1(E)$ of an $O(n)$-bundle $E$ vanishes if and only if there exists an $O(n)^c$-equivariant function $E^c\to\mathbb Z/2$ that extends $O(n)^c\to\mathbb Z/2$.
Since the topology of $\mathbb Z/2$ is discrete, this is equivalent to asking for the existence of a $\mathbb Z/2$-equivariant function $\pi_0(E)\cong\pi_0(E^c)\to\mathbb Z/2$, that extends the canonical isomorphism $\pi_0(O(n))\cong \pi_0(O(n)^c)\cong\mathbb Z/2$. This is in turn equivalent to asking for the existence of an $O(n)$-equivariant function $E\to\mathbb Z/2$ that extends the determinant map $O(n)\to\mathbb Z/2$.  This argument (re)proves that well-known fact that the first Stiefel-Whitney class vanishes if and only if there exists a coherent choice of sign for the bundle $E$, which is one of the usual definitions of orientability.
\end{ex}

\section{The plus-cohomology groups of a principal bundle}

Given a topological group $G$ whose homotopy groups are countable, $Z$ a countable abelian group, and $E$ a $G$-bundle, we exploit from Theorems \ref{groupreductioncor} and \ref{retractioncor} the idea that equivariant maps $E^c\to M^{k-1}Z$ enumerate the group reductions of a certain bundle $E$ with respect to some universal characteristic class in $H^k(BG;Z)$, and define an invariant that implements this point of view.

\subsection{The category of pointed bundles}

\label{The category of pointed bundles}
This section provides the correct framework to study the equivariant maps $E^c\to M^{k-1}Z$.

\vphantom{}

We say that a $G$-bundle $E$ over a pointed space $X$ is \textbf{pointed} if $E$ is endowed with a base point that makes the projection $E\to X$ pointed. Given $E$ and $E'$ two pointed $G$-bundles over $X$, we write $E\cong_*E'$ if there exists a pointed equivalence of $G$-bundles over $X$ between them. This clearly defines an equivalence relation.

\vphantom{}

The following is a variant of Theorem \ref{classificationnice} that takes into account the base points. The proof of the improved statement can be found in \cite{rovellitesi}. Given maps of countable CW-groups $a,b\colon G\to G'$ we say that they are \textbf{based algebraically equivalent} if $\tilde Ba\simeq_*\tilde Bb\colon \tilde BG\to\tilde BG'$. In this case we write $a\equiv_*b$.
\begin{teorema}[\cite{rovellitesi}]
\label{pointedclassificationnice}
Let $X$ be a connected countable CW complex pointed in a vertex and $G$ a countable CW-group. Then the assignments
$(-)_*(\tilde PX)$ and $(-)^*(\tilde EG)$
induce bijections
$$\mathcal Gp(\mathcal Top)(\tilde\Omega X,G)/_{\equiv_*}\ \cong\ _X\mathcal Bun_G/_{\cong_*}\ \cong\ \mathcal Top(X,\tilde BG)/_{\simeq_*}.$$
\end{teorema}

\begin{dfnt}
We denote by $\mathcal Bun_*$ the \textbf{category of pointed principal bundles}.
\begin{itemize}
\item An object consists of a triple $(X,G,E)$, where $X$ is a pointed topological space, $G$ is a topological group and $E$ is a pointed $G$-bundle over $X$. The space $X$ is the \textbf{base space}, the group $G$ is the \textbf{structure group} and the space $E$ is the \textbf{total space}. We denote this object by ${}_XE_G$, or just by $E$ when there is no risk of confusion.
\item A morphism ${}_XE_G\to{}_{X'}E'_{G'}$ consists of a triple $(\psi,g,a)$, where $g\colon X\to X'$ is a pointed map, $a\colon G\to G'$ is a homomorphism of topological groups, and $\psi\colon E\to E'$ is a pointed map that is $a$-equivariant and $g$-coequivariant. The pointed map $g$ is called the \textbf{geometric information}, the homomorphism $a$ is called \textbf{algebraic information}, and the pointed map $\psi$ is called \textbf{total information}. We denote this object by ${}_g\psi_a$, or just $\psi$ when there is no risk of confusion.
\item Composition and identity are defined in the obvious way.
\end{itemize}
\end{dfnt}
\label{pointedframework}
Restricting ourselves to pointed bundles allows a better interaction between the category of bundles and the categories of groups and spaces we work with. For instance, if $E$ is a pointed $G$-bundle over $X$, there is a canonical $G$-equivariant inclusion $G\to E$ determined by sending the identity of $G$ to the base point of $E$.

\begin{rem}
\label{mapsofbundles}
For every pointed bundle ${}_XE_G$, the canonical inclusion $\iota$ of the structure group and the projection $\pi$ on the base space are morphisms in $\mathcal Bun_*$, as displayed:
$$\xymatrix{{}_* G_G\ar[r]^-{{}_*\iota_G}&{}_XE_G\ar[r]^-{{}_X\pi_*}&{}_XX_*}.$$
Every morphism of bundles of the form ${}_X\psi_a\colon {}_XE_G\to{}_{X'}E'_{G'}$ induces $a$ on the distinguished fibers and $g$ on the base spaces, namely, there is a commutative diagram in $\mathcal Bun_*$:
$$\xymatrix@R=.4cm{
{}_* G_G\ar[r]^-{{}_*\iota_G}\ar[d]_-{{}_* a_a}&{}_XE_G\ar[r]^-{{}_X\pi_*}\ar[d]_-{{}_g\psi_a}&{}_XX_*\ar[d]^-{{}_gg_*}\\
{}_* {G'}_G'\ar[r]^-{{}_*\iota'_{G'}}&{}_{X'}E'_{G'}\ar[r]^-{{}_{X'}\pi'_*}&{}_{X'}X'_*.
}$$
\end{rem}

Using Remark \ref{mapsofbundles}, one can check that topological groups and pointed spaces embed fully faithfully into pointed bundles.
\begin{prop}
\label{embedding}
\begin{enumerate}
\item The assignment $X\mapsto{}_XX_*,$
which interprets a pointed space as a \textbf{discrete bundle} over itself with trivial structure group,
induces a fully faithful inclusion\linebreak$\mathcal Top_*\hookrightarrow\mathcal Bun_*$.
\item The assignment $G\mapsto{}_* G_G$,
which interprets a topological group as a \textbf{codiscrete $G$-bundle} over a point, induces a fully faithful inclusion $\mathcal Gp(\mathcal Top)\hookrightarrow\mathcal Bun_*$.
\label{homgroup}
\end{enumerate}
\end{prop}

\begin{rem}
\label{puppesequence}
Let $E$ be a pointed $G$-bundle over $X$, with $X$ a countable CW-complex pointed in a vertex
and $G$ a countable CW-group.
By Theorem \ref{pointedclassificationnice}, there exist pointed classifying maps
$$t\colon X\to \tilde BG\text{ and }b\colon \tilde\Omega X\to G$$
that yield pointed equivalences of $G$-bundle over $X$
$$t^*\tilde EG\cong_*E\cong_*b_*\tilde PX.$$
The maps $t$ and $b$ are unique up to based homotopy and based algebraic equivalence, respectively.
Having chosen two such maps $b$ and $t$, we can build the \textbf{Nomura-Puppe sequence} of the bundle
$$
\xymatrix{\tilde\Omega X\ar[r]^-b&G\ar[r]^-{\iota}&E\ar[r]^-{\pi}&X\ar[r]^-t&\tilde BG}
$$
which, thanks to Proposition \ref{embedding}, lives in the category $\mathcal Bun_*$ as
$$\xymatrix{{}_*\tilde\Omega X_{\tilde\Omega X}\ar[r]^-{{}_* b_b}&{}_* G_G\ar[r]^-{{}_*\iota_G}&{}_XE_G\ar[r]^-{{}_X\pi_*}&{}_XX_*\ar[r]^-{{}_tt_*}&{}_{\tilde BG}\tilde BG_*}.$$
\end{rem}

Using an Eckmann-Hilton argument, one can check that, for any abelian topological group $A$, the codiscrete $A$-bundle over a point $A\mapsto{}_* A_A$ is an abelian group object in $\mathcal Bun_*$.
As a consequence, for any topological abelian group $A$ the assignment ${}_XE_G\mapsto{\mathcal Bun_*}({}_XE_G,{}_* A_A)=:{\mathcal Bun_*}(E,A)$
defines a contravariant functor in abelian groups
$${\mathcal Bun_*}(-,A)\colon {\mathcal Bun_*}^{\textnormal{op}}\to\mathcal Ab,$$
where the group structure is given by pointwise multiplication.
By applying the functor $\mathcal Bun_*(-,A)$ to the Nomura-Puppe sequence from Remark \ref{puppesequence}, we obtain a diagram of abelian groups,
$$
\label{puppenonexactsequence}
\xymatrix{{\mathcal Bun_*}(\tilde BG,A)\ar[r]^-{t^*}&{\mathcal Bun_*}(X,A)\ar[r]^-{\pi_*}&{\mathcal Bun_*}(E,A)\ar[r]^-{\iota^*}&{\mathcal Bun_*}(G,A)\ar[r]^-{b^*}&{\mathcal Bun_*}(\tilde\Omega X,A),}
$$
which needs not be exact, and not even a complex of abelian groups.

\vphantom{}

We will define an equivalence relation $\simeq_+$ on ${\mathcal Bun_*}(E,A)$ with the following properties.
\begin{itemize}
\item The \textbf{plus-homgroup} $[E,A]_+$, defined as the quotient
$$[E,A]_+:={\mathcal Bun_*}(E,A)/_{\simeq_+},$$
is an abelian group with respect to pointwise multiplication.
\item When $A=M^kZ\simeq K(Z;k)$ is a $k$-fold Milgram delooping and $E=X$ is a discrete bundle or $E=G^{\text{c}}$ is a codiscrete bundle, the plus-homgroup recovers the ordinary cohomology (of the classifying space):
$$[X,M^kZ]_+\cong H^k(X;Z)\text{ \quad and\quad }[G^{\textnormal{c}},M^kZ]_+\cong H^{k+1}(\tilde BG;Z).$$
\item There is an exact sequence of abelian groups obtained by quotienting the diagram above,
$$
\xymatrix{[BG,A]_+\ar[r]^-{t^*}&[X,A]_+\ar[r]^-{\pi_*}&[E,A]_+\ar[r]^-{\iota^*}&[G,A]_+\ar[r]^-{b^*}&[\tilde\Omega X,A]_+}.
$$
\end{itemize}

\begin{rem}
\label{classifyingmap}
Let $G$ be a topological group, and $\tilde E$ a pointed universal bundle for $G$ over the pointed space $\tilde B$,
and $E$ a pointed $G$-bundle over a CW-complex $X$ pointed in a vertex. By a pointed variant of Theorem \ref{classificationGbundles} that takes into account the base point (and can be found in \cite{rovellitesi}), there exists a pointed classifying map $t\colon X\to\tilde B$ that induces a pointed equivalence $t^*\tilde E\cong_* E$. Such a map $t$ is unique up to based homotopy.
The map $t$ determines further data. More precisely, the following data are equivalent.
\begin{enumerate}
\item A pointed map $t\colon X\to\tilde B$ together with an equivalence $t^*\tilde E\cong E$;
\item A pointed $G$-equivariant map $T\colon E\to\tilde E$;
\item A pullback diagram of pointed spaces
$$\xymatrix@R=.4cm{E\pb\ar@{->>}[r]^-{\pi}\ar[d]_-T&X\ar[d]_-t\\\tilde E\ar@{->>}[r]_-{\tilde\pi}&\tilde B.}$$
\end{enumerate}
In this case, we refer to the map $T\colon E\to\tilde E$ as a \textbf{universal map} for $E$ and to the map $t\colon X\to\tilde B$ as a \textbf{pointed classifying map} for $E$ (with respect to $\tilde E$).
Note that both maps live in $\mathcal Bun_*$, as displayed
$${}_tT_G\colon {}_XE_G\to{}_{\tilde B}\tilde E_G\text{ and }{}_tt_*\colon {}_{\tilde B}\tilde B_*\to{}_{X}X_*.$$
\end{rem}

\subsection{Some distinguished classes of maps}

We now identify certain classes of maps in $\mathcal Bun_*$ which will play a fundamental role in defining the plus-cohomology groups. Most of these maps of bundles will be nullhomotopic when regarded as maps of spaces.

\begin{dfnt}
\label{typeofmorphisms}
Let $X$ be a CW-complex pointed in a vertex,
$G$ a countable CW-group, $E$ a pointed $G$-bundle over $X$, and $A$ a topological abelian group.
\begin{itemize}
	\item A morphism $\psi\colon E\to A$ in $\mathcal Bun_*$ is \textbf{of type U} if it factors through a universal map $T\colon E\to\tilde E$ in $\mathcal Bun_*$, where $\tilde E$ is a pointed $G$-bundle over a CW-complex $\tilde B$ pointed in one of its vertices such that $\tilde E$ admits the structure of a countable CW-complex and is contractible. Explicitly, this means that $\psi$ factors as
$${}_*\psi_a\colon \xymatrix{{}_XE_G\ar[r]^-{{}_tT_G}&{}_{\tilde B}\tilde E_G\ar[r]^-{{}_*\kappa_a}&{}_* A_A}$$
for some $\kappa\colon \tilde E\to A$ in ${\mathcal Bun_*}$.
	\item A morphism $\psi\colon E\to A$ in $\mathcal Bun_*$ is \textbf{of type P} if it factors though the projection $\pi\colon E\to X$ in $\mathcal Bun_*$, or equivalently if the algebraic information is the constant map at the neutral element $e_A$ of $A$. Explicitly, this means that that $\psi$ factors as
$${}_*\psi_*\colon \xymatrix{{}_XE_G\ar[r]^-{{}_\pi\pi_*}&{}_XX_*\ar[r]^-{{}_* s_*}&{}_* A_A}$$
for some $s\colon X\to A$ in ${\mathcal Bun_*}$.
	\item A morphism $\psi\colon E\to A$ in $\mathcal Bun$ is \textbf{of type UP} if it factors through $\tilde\pi\circ T=t\circ\pi\colon E\to\tilde B$ in ${\mathcal Bun_*}$, where $\tilde\pi\colon\tilde E\to\tilde B$ the projection of a pointed universal bundle that admits the structure of a countable CW-complex, $\tilde B$ is a good classifying space, $T\colon E\to \tilde E$ is a universal map and $t\colon X\to\tilde B$ is the corresponding pointed classifying map. Explicitly, this means that that $\psi$ factors as
$${}_*\psi_*\colon \xymatrix@R=.2cm{
{}_XE_G\ar@{=}[d]\ar[r]^-{{}_\pi\pi_*}&{}_XX_*\ar[r]^-{{}_tt_*}&{}_{\tilde B}\tilde B_*\ar[r]^-{{}_* s'_*}\ar@{=}[d]&{}_* A_A\ar@{=}[d]\\
{}_XE_G\ar[r]_-{{}_tT_G}&{}_{\tilde B}\tilde E_G\ar[r]_-{{}_{\tilde\pi}\tilde\pi_*}&{}_{\tilde B}\tilde B_*\ar[r]_-{{}_* s'_*}&{}_* A_A}$$
for some $s'\colon \tilde B\to A$ in ${\mathcal Bun_*}$.
As proven in Proposition \ref{propertiestypeofmorphisms}(4), this is equivalent to saying that $\psi$ is of type U and of type P, which justifies the terminology.
	\item A morphism $\psi\colon E\to A$ in $\mathcal Bun$ is \textbf{weakly of type UP} or \textbf{of type wUP} if it is of type P and the factor $s\colon X\to A$ occurring in the factorization $\psi=s\circ\pi$ factors up to based homotopy through a pointed classifying map $t\colon X\to\tilde B$, where $\tilde B$ is good pointed classifying space. In particular, this means that $\psi$ factors (up to based homotopy) as 
$${}_*\psi_*\colon \xymatrix@R=.2cm{{}_XE_G\ar@{=}[d]\ar[r]^-{{}_X\pi_*}&{}_XX_*\ar@{=}[d]\ar[rr]^-{{}_* s_*}&&{}_* A_*\ar@{=}[d]\\
{}_XE_G\ar[r]^-{{}_\pi\pi_*}&{}_XX_*\ar[r]^-{{}_tt_*}&{}_{\tilde B}\tilde B_*\ar@{}[u]|-{\simeq_*}\ar[r]^-{{}_* s'_*}&{}_* A_A}$$
for some $s\colon X\to A$ and $s'\colon \tilde B\to A$. As pointed out in Proposition \ref{propertiestypeofmorphisms}(3), this is a weaker condition than being of type UP, which justifies the terminology.
\end{itemize}
\end{dfnt}

\begin{rem}
\label{remtrace}
Every map of type U factors through a contractible space, and every map of type wUP is homotopic to a map of type UP, which factors through a contractible space. It follows that maps of types U and wUP are nullhomotopic when the bundle structure is disregarded and they are considered maps of spaces.
\end{rem}

We now prove useful closure properties of these distinguished classes of maps.
\begin{prop}
\label{propertiestypeofmorphisms}
Let $X$ be a CW-complex pointed in a vertex,
$G$ a countable CW-group, $E$ a pointed $G$-bundle over $X$ and $A$ an abelian topological group. Let $\psi$ and $\psi'$ be morphisms $E\to A$ in $\mathcal Bun_*$, and $\psi\cdot\psi'$ their pointwise multiplication.
\begin{enumerate}
	\item If $\psi$ and $\psi'$ are of type U, then $\psi\cdot\psi'$ is of type U.
	\item If $\psi$ and $\psi'$ are of type wUP, then $\psi\cdot\psi'$ is of type wUP.
	\item If $\psi$ is of type UP, then $\psi$ is of type wUP.
	\item $\psi$ is of type UP if and only if it is of type P and of type U.
\end{enumerate}
\end{prop}
\begin{proof}
We prove (1). As $\psi$ and $\psi'$ are of type U, by definition they factor through pointed universal bundles $\tilde E$ and $\tilde E'$ that admit the structures of countable CW-complexes,
	as displayed:
$$\psi\colon \xymatrix{E\ar[r]^-T&\tilde E\ar[r]^-\kappa&A}\text{ and }\psi'\colon \xymatrix{E\ar[r]^-{T'}&\tilde E'\ar[r]^-{\kappa'}&A.}$$
Then the product $\psi\cdot\psi'$ factors through $\tilde E\times\tilde E'$, as displayed
$$\psi\cdot\psi'\colon \xymatrix{E\ar[r]^-{(T,T')}&\tilde E\times\tilde E'\ar[r]^-{\kappa\times \kappa'}&A\times A\ar[r]^-\mu&A.}$$
We note that the pointed space $\tilde E\times\tilde E'$ is a pointed universal bundle for $G$ that admits the structure of a countable CW-complex. Indeed, it is contractible, and is a pointed $G$-bundle over $\tilde E\otimes_G\tilde E'$, which is a bundle over $\tilde B$ with fiber $\tilde E'$. Then $\tilde E\otimes_G\tilde E'$ is a countable CW-complex by Proposition \ref{totalspaceCWcomplex}, and $\tilde E\times\tilde E'$ is a countable CW-complex because it is a product of countable CW-complexes.

We prove (2). As $\psi$ and $\psi'$ are of type wUP, by definition they factor through the base space $X$, as displayed:
$$\psi\colon \xymatrix{E\ar[r]^-{\pi}&X\ar[r]^-s&A}\text{ and }\psi'\colon \xymatrix{E\ar[r]^-{\pi}&X\ar[r]^-{s'}&A.}$$
Then the product $\psi\cdot\psi'$ factors through $X$, as displayed
$$\psi\cdot\psi'\colon \xymatrix{E\ar[r]^-{\pi}&X\ar[r]^-{(s,s')}&A\times A\ar[r]^-\mu&A.}$$
Moreover, without loss of generality, $s$ and $s'$ factor up to based homotopy through some good classifying space $\tilde B$, as displayed:
$$s\colon \xymatrix{X\ar[r]^-{t}&\tilde B\ar[r]^-{\overline s}&A}\text{ and }s'\colon \xymatrix{X\ar[r]^-{t'}&\tilde B'\ar[r]^-{\overline s'}&A.}$$
As $\tilde E$ is a $G$-bundle and $\tilde E'$ is a universal bundle for $G$, there exists a pointed map $\tau\colon \tilde B\to\tilde B'$ such that
$$\tau^*\tilde E'\cong_*\tilde E.$$
We then have the following pointed equivalences of $G$-bundles over $X$:
$$(\tau\circ t)^*\tilde E'\cong_*t^*\tau^*\tilde E'\cong_*t^*\tilde E\cong_*E\cong_*t'^*\tilde E',$$
which implies that
$$t'\simeq_*\tau\circ t.$$
In other words, up to based homotopy, $s'\colon X\to A$ also factors through $\tilde B$, as
$$[s']\colon\xymatrix{X\ar[r]^-{t}&\tilde B\ar[r]^-{\tau}&\tilde B'\ar[r]^-{\overline s'}&A.}$$
Then the product $s\cdot s'$ factors up to based homotopy through $\tilde B$, as
$$[s\cdot s']\colon\xymatrix@C=1.5cm{X\ar[r]^-{t}&\tilde B\ar[r]^-{(\overline s,\overline s'\circ\tau)}&A\times A\ar[r]^-\mu&A.}$$
It follows that $\psi\cdot\psi'$ is of type wUP.

The statement (3) is straightforward from the definitions. We prove (4). The direct implication is straightforward. We prove the inverse direction. As $\psi$ is of type U, by definition $\psi$ factors through a universal bundle $\tilde E$ that admits the structure of a countable CW-complex as displayed
$$\psi\colon \xymatrix{
E \ar[r]^-T&\tilde E\ar[r]^-{\kappa}&A}.$$
	Then, as $\psi$ is of type P, the algebraic information $a(\kappa)$ of $\kappa$ vanishes, and we have that
$$0=a(\psi)=a(\kappa\circ T)=a(\kappa)\circ a(T)=a(\kappa)\circ\id{}_G=a(\kappa).$$
So the map $\kappa$ is of type P, and factors therefore through the quotient $\tilde E/G=\tilde B$.
Thus we can write $\psi$ as follows

$$\psi\colon\xymatrix@R=.35cm{E\ar@/^1.5pc/[rr]\ar[r]^-T\ar@{->>}[d]&\tilde E\ar[r]^-\kappa\ar@{->>}[d]&A\\
X\ar[r]_-t&\tilde B\ar[ru]&}$$
and see that it factors through the classifying space $\tilde B$. It follows that $\psi$ is of type UP.\qedhere
\end{proof}

\subsection{The plus-cohomology groups}
In this section, we define the plus-cohomology-groups of a bundle. Denote by $\psi^{-1}$ and by $\psi\cdot\psi'$ the pointwise inverse and the pointwise multiplication of maps $E\to A$, respectively.
\begin{dfnt}
Let $X$ be CW-complex pointed in a vertex,
$G$ a countable CW-group, $E$ a pointed $G$-bundle over $X$ and $A$ an abelian topological group.
\begin{itemize}
\item A \textbf{UP-decomposition} of $\psi\colon E\to A$ in ${\mathcal Bun_*}$ consists of a morphism $\psi_U\colon E\to A$ of type U and a morphism $\psi_P\colon E\to A$ of type P such that
$$\psi=\psi_U\cdot\psi_P\in{\mathcal Bun_*}(E,A).$$
\item A \textbf{UW-decomposition} of $\psi\colon E\to A$ in ${\mathcal Bun_*}$ consists of a morphism $\psi_U\colon E\to A$ of type U and a morphism $\psi_W\colon E\to A$ of type wUP such that
$$\psi=\psi_U\cdot\psi_W\in{\mathcal Bun_*}(E,A).$$
\item Two morphisms $\psi,\psi'\colon E\to A$ in ${\mathcal Bun_*}$ are \textbf{plus-equivalent}, and we write $\psi\simeq_+\psi'$, if and only if $\psi\cdot\psi'^{-1}\colon E\to A$ admits a UW-decomposition
$$\psi\cdot\psi'^{-1}=\psi_U\cdot\psi_W\in{\mathcal Bun_*}(E,A).$$
\end{itemize}
\end{dfnt}

Note that he relation $\simeq_+$ is an equivalence relation. Indeed, it is easily verified to be reflexive and symmetric, and is transitive as a consequence of Proposition \ref{propertiestypeofmorphisms} (1) and (2) and of the fact that $A$ is abelian. It is also a compatible with pointwise multiplication, again thanks to Proposition \ref{propertiestypeofmorphisms} (1) and (2). 

\begin{notation}
We denote by $[\psi]_+$ the class of $\psi\colon E\to A$ with respect to the relation $\simeq_+$, and by $[E,A]_+$ the quotient
$$[E,A]_+:={\mathcal Bun_*}(E,A)/_{\simeq_+},$$
which is an abelian group with respect to pointwise multiplication.

\vphantom{}

Recall the notation $E^c$ is from Notation \ref{notationc}. If $Z$ is a countable abelian group and $k\ge0$, we define the \textbf{plus-cohomology group} $H^k_+(E;Z)$ by
$$H^k_+(E;Z):=[E^{\text{c}},M^kZ]_+.$$
\end{notation}

\begin{rem}
Let $X$ be a connected space pointed in $x_0$, and $E$ an unpointed bundle over $X$. Given two points $e_0$ and $e_0'$ in the fiber $E|_{x_0}$, there exists a pointed equivalence
$$(E,e_0)\cong_*(E,e_0'),$$
which can be realized as follows. If $g_0$ is the element of $G$ such that $e_0'=e_o\cdot g_0$, then the map
$$\xymatrix{E\cong E\times0\ar[r]^-{E\times g_0}&E\times G\ar[r]^-\cdot&E}$$
is an equivalence and maps $e_0$ to $e_0'$.
Therefore, the pointed equivalence class of $E$ does not depend on the chosen base point. In particular, it makes sense to talk about the plus-cohomology of $E$ without specifying any base point for $E$.
\end{rem}

The notion of plus-equivalence generalizes both the notions of based algebraic homotopy for pointed maps and of based algebraic equivalence for homomorphisms of topological groups.
If $A$ is a topological group we denote by $e_A$ both the neutral element of $A$ and any map into $A$ constant at the value $e_A$.
\begin{prop}
\label{degeneratecase}
Let $A$ be an abelian countable CW-group.
 \begin{enumerate}
	\item Given a CW-complex $X$ pointed in a vertex, and pointed maps $s,s'\colon X\to A$ we have that
	$$s\simeq_+s'\colon X\to A\text{ if and only if }s\simeq_* s'\colon X\to A.$$
	\item Given a countable CW-group $G$, and maps of topological groups $a,a'\colon G\to A$ we have that
	$$a\simeq_+a'\colon G\to A\text{ if and only if }a\equiv_* a'\colon G\to A.$$
\end{enumerate}
\end{prop}

The following lemma is preliminary to Proposition \ref{degeneratecase}.
\begin{lem}
\label{lemmaEG}
For a map of topological groups $a\colon G\to A$ between countable CW-groups the following are equivalent.
\begin{enumerate}
\item The homomorphism $a$ extends to an $a$-equivariant pointed map $\kappa\colon \tilde EG\to A\text{ in }\mathcal Bun_*$.
\item There is a pointed equivalence $a_*\tilde EG\cong_*\tilde BG\times A$ of $A$-bundles over $\tilde BG$.
\item The homomorphism $a$ is based equivalent to the constant map, i.e., $a\equiv_* e_A\colon G\to A$.
\end{enumerate}
\end{lem}
\begin{proof}[Proof of Lemma \ref{lemmaEG}]
Recall that every pointed $A$-equivariant map between $A$-bundles over $\tilde BG$ that is compatible with the projection onto $\tilde BG$ is a pointed equivalence of $A$-bundles over $\tilde BG$.

We prove $[(1)\Rightarrow(2)]$. Since $\kappa$ is $a$-equivariant, it induces an $A$-equivariant map $\overline\kappa\colon a_*\tilde EG\to A$ and factors through $a_*\tilde EG$ as displayed:
$$\kappa\colon \xymatrix@C=2cm{\tilde EG\cong_*\tilde EG\otimes_GG\ar[r]^-{\tilde EG\otimes_Ga}&\tilde EG\otimes_GA=a_* \tilde EG\ar[r]^-{\overline\kappa}&A}.$$
Moreover, the projection of $a_* \tilde EG$ gives a map $\pi\colon a_* \tilde EG\to\tilde BG$.
Combining $\pi$ and $\overline\kappa$, we obtain a pointed equivalence of $A$-bundles over $\tilde BG$
$$(\pi,{\overline\kappa})\colon \xymatrix{a_* \tilde EG\ar[r]^-{\cong_*}&\tilde BG\times A.}$$

We prove $[(1)\Leftarrow(2)]$. Given a pointed equivalence $\phi\colon a_* \tilde EG\cong_*\tilde BG\times A$ of $A$-bundles over $\tilde BG$, the following map $\kappa$ is equivariant and extends the map $a$:
	$$\kappa\colon \xymatrix@C=1.3cm{\tilde EG\cong_*\tilde EG\otimes_GG\ar[r]^-{\tilde EG\otimes_Ga}&\tilde EG\otimes_GA=a_* \tilde EG\ar[r]^-\phi_-{\cong_*}&\tilde BG\times A\ar[r]^-{\pr_2}&A.}$$

We prove $[(2)\Leftrightarrow(3)]$. By Theorem \ref{pointedclassificationnice} there is a pointed equivalence between 
$\tilde BG\times A={e_A}_*\tilde EG$ and $a_*\tilde EG$ if and only if $a$ is based equivalent to the constant homomorphism $e_A\colon G\to A$.
\qedhere
\end{proof}

We can now prove the proposition.

\begin{proof}[Proof of Proposition \ref{degeneratecase}]
\begin{enumerate}
	\item We prove the direct implication. If $s\simeq_* s'\colon X\to A$, then\linebreak$s'\cdot s^{-1}\simeq_* e_A\colon X\to A$ factors up o based homotopy through a point, therefore $s'\cdot s^{-1}$ is of type wUP in ${\mathcal Bun_*}(X,A)$. So $s\simeq_+ s'$.

We prove the inverse implication. If $s\simeq_+ s'\colon X\to A$, then $s'\cdot s^{-1}=\psi_U\cdot\psi_W$ where $\psi_U\colon X\to A$ is of type U and $\psi_W\colon X\to A$ is of type wUP in ${\mathcal Bun_*}(X,A)$.
By Remark \ref{remtrace}, all maps of type U and all maps of type wUP are based homotopic to the constant pointed map, and therefore
$$s'\cdot s^{-1}=\psi_U\cdot\psi_W\simeq_* e_A\cdot e_A=e_A\colon X\to A.$$
It follows that $s\simeq_* s'$.
	\item We prove the direct implication. If $a\equiv_* a'\colon G\to A$, then $a'\cdot a^{-1}\equiv_* e_A\colon G\to A$. By Lemma \ref{lemmaEG}, $a'\cdot a^{-1}$ factors through $\tilde EG$, and is therefore of type U in ${\mathcal Bun_*}(G,A)$. It follows that $a\simeq_+a'$.

We prove the inverse implication. If $a\simeq_+ a'\colon G\to A$, then $a'\cdot a^{-1}=\psi_U\cdot\psi_W$ where $\psi_U\colon G\to A$ is of type U and $\psi_W\colon G\to A$ is of type wUP in ${\mathcal Bun_*}(G,A)$.
As $\psi_U$ is of type U, it factors through some classifying space $\tilde E$ for $G$. As $\tilde EG$ is a $G$-bundle and $\tilde E$ is a universal bundle for $G$, there exists a $G$-equivariant pointed map $EG\to\tilde E$.
Then $\psi_U$ also factors through $EG$, as displayed
$$\psi_U\colon\xymatrix@R=.3cm{\ar@{=}[d]G\ar@/^1.5pc/[rr]\ar@{^{(}->}[r]&\tilde E\ar[r]_-{\tilde\kappa}&A.\\
G\ar@{^{(}->}[r]&EG\ar[u]&}$$
By Lemma \ref{lemmaEG}, $\psi_U$ is then based algebraic equivalent to the constant homomorphism.
On the other hand, $\psi_W$ is constant because it factors through the base space of $G$, which is a point. Therefore
$$a'\cdot a^{-1}=\psi_U\cdot\psi_W=\psi_U\cdot e_A=\psi_U\equiv_*e_A\colon G\to A.$$
It follows that $a\equiv_* a'$.\qedhere
\end{enumerate}
\end{proof}

When $E=X$ is a discrete bundle, or $E=G^{\text{c}}$ is a codiscrete bundle, the plus-cohomology group recovers ordinary (group) cohomology.
\begin{prop}
\label{degeneratecasegroups}
Let $X$ be a CW-complex pointed in a vertex, $G$ a topological group that admits a countable CW-replacement, and $A$ an abelian countable CW-group. The following isomorphisms of abelian groups hold.
\begin{displaymath}
  \begin{array}{lrcl}
    1) & [{}_XX_*,{}_* A_A]_+&{}\cong{}&[X,A]_*\\
    2) & [{}_* G^{\text{c}}_{G^{\text{c}}},{}_* A_A]_+&{}\cong{}&[BG,MA]_*\\
    3) & [{}_X{X\times G^{\textnormal{c}}}_{G^{\textnormal{c}}},{}_X A_A]_+&{}\cong{}&[X,A]_*\oplus[BG,MA]_*\\
    4) &[{}_{\tilde BG}\tilde EG_{G},{}_X A_A]_+&{}\cong{}&0
  \end{array}
\end{displaymath}

\end{prop}

We need a preliminary lemma to prove (2).
\begin{lem}
\label{connectingmap}
For every abelian countable CW-group $A$ and any connected countable CW-complex $X$ pointed in a vertex there is an isomorphism of abelian groups
$$\xymatrix{[\tilde\Omega X,A]_+\ar[r]^-\cong^-{}&[X,MA]_+}.$$
\end{lem}
\begin{proof}[Proof of Lemma \ref{connectingmap}]
Let $\tau\colon \tilde BA\to MA$ be the natural based homotopy equivalence of Theorem \ref{classifyingspacemilgram}, and let $\eta_X\colon X\to\tilde B\tilde\Omega X$ be the classifying map of $\tilde PX$ over the universal bundle $\tilde E\tilde\Omega X$.
One can check that the map
$$\tau\circ\tilde B-\circ\eta_X\colon[\tilde\Omega X,A]_+\to [X,MA]_+$$
realizes the bijection
$$[\tilde\Omega X,A]_+\cong\mathcal Gp(\mathcal Top)(\tilde\Omega X,A)/_{\equiv}\cong\mathcal Top(X,\tilde BA)/_{\simeq}\cong\mathcal Top(X,MA)/_{\simeq}\cong [X,MA]_+,$$
obtained by means of Proposition \ref{degeneratecase}, Theorem \ref{classificationnice}, and Theorem \ref{classifyingspacemilgram}.
The map is also a group homomorphism. Indeed, for any $a,a'\colon \tilde\Omega X\to A$, the following diagram commutes up to based homotopy:
$$\xymatrix@R=.05cm@C=1.5cm{X\ar@/_3pc/[rrdddd]_(.3){(\tau\circ\tilde Ba\circ\eta_X,\tau\circ Ba'\circ\eta_X)}\ar[r]^-{\eta_X}&\tilde B\tilde\Omega X\ar@/^2pc/[rrrd]|-{\tilde B(a\cdot a')}\ar[rdd]|-{(Ba,Ba')}\ar[rrd]_-{B(a,a')}\ar@/_2pc/[rdddd]|-{(\tau\circ \tilde Ba,\tau\circ \tilde Ba')}&&&\\
&&&\tilde B(A\times A)\ar[dd]^-{\tau_{A\times A}}\ar[ld]^-{(\tilde B\pr_1,\tilde B\pr_2)}\ar[r]_-{B\mu_{A}}&\tilde BA\ar[dd]^-{\tau_A}\\
&&\tilde BA\times \tilde BA\ar[dd]_-{\tau_A\times\tau_A}&&\\
&&&M(A\times A)\ar@{=}[ld]\ar[r]_-{M\mu_{A}}&MA.\\
&&MA\times MA\ar@/_3pc/[rru]_-{\mu_{MA}}&&\\
&&&&\\
&&&&\square\qedhere\\
}$$
\end{proof}

We can now prove the proposition. Note that the proofs of (1), (2) and (4) follow from previous results. However, there is no elementary proof of (3). We here give two arguments, but they both use results that will appear later in the paper.

\begin{proof}
Statement (1) is a consequence of Proposition \ref{degeneratecase} and Statement (2) is a consequence of (1) and Lemma \ref{connectingmap}. Statement (4) is a consequence of the fact that, by definition, every map $EG\to A$ is of type U, and therefore plus-equivalent to the trivial map.

We prove (3). By Lemma \ref{lemmaprecomposition}, the assignments
$(s\colon X\to A)\mapsto(s\circ\pr{}_1\colon X\times G\to A)\text{ and }(a\colon G\to A)\mapsto(a\circ\pr{}_2\colon X\times G\to A)$
induce maps $[X,MA]_+\to [X\times G,A]_+\text{ and }[G,A]_+\to [X\times G,A]_+.$
They can be combined into a map $[X,MA]_+\oplus[BG,MA]_+\to [X\times G,A]_+$
which is given by $(s\colon X\to A,a\colon G\to A)\mapsto((s\circ\pr{}_1)\cdot(a\circ\pr{}_2)\colon X\times G\to A).$
One can prove directly that the map is a bijection. Alternatively, by using Theorem \ref{characteristicexactsequence} we obtain a short exact sequence of abelian groups
$$\xymatrix{\{0\}\ar[r]&[X,A]_+\ar[r]^-{\pi^*}&[X\times G,A]_+\ar[r]^-{\iota^*}&[G,A]_+\ar[r]&\{0\},}$$
which splits because the mentioned above map $[G,A]_+\to [X\times G,A]_+$ is a section.
\end{proof}

By specializing the last proposition to the case $A=M^kZ$, we obtain the following.

\begin{cor} 
\label{degeneratecasegroupscor}
Let $X$ be a CW-complex pointed in a vertex, $G$ a group that admits a countable CW-replacement, $Z$ a countable abelian group, and $k>0$. The following isomorphisms of abelian groups hold.
\begin{displaymath}
  \begin{array}{lrcl}
    1) & H^k_+({}_XX_*,Z)_+&{}\cong{}&H^k(X;Z)\\
    2) &H^k_+({}_* G_{G},Z)&{}\cong{}&H^{k+1}(BG;Z)\\
    3) &H^k_+({}_X X\times G_{G},Z)&{}\cong{}&H^{k+1}(BG;Z)\oplus H^k(X;Z)\\
    4) &H^k_+({}_{BG} EG_{G},Z)&{}\cong{}&0
  \end{array}
\end{displaymath}
\end{cor}
As a consequence of Remark \ref{remtrace}, if two maps of bundles are plus-equivalent, they they are also based homotopic as pointed maps of pointed spaces. This can be summarized by the following fact.

\begin{rem}
\label{trace}
For every pointed $G$-bundle $E$ over $X$, the identity yields a homomorphism of abelian groups
 $$[{}_XE_G,{}_* A_A]_+\to[E,A]_*.$$
By specializing to the case of $A=M^kZ$, and using the fact that the construction $(-)^c$ does not change the homotopy type, the map becomes
 $$H^k_+(E;Z)\to H^k(E^{\text{c}};Z)\cong H^k(E;Z).$$
This map is an isomorphism if $E=X$ is a discrete bundle. However it will not be in general, not even when the bundle is codiscrete.
\end{rem}

We now discuss the functoriality of the assignment $E\mapsto[E,A]_+$.
Although the group $\mathcal Bun_*(E,A)$ is functorial in $E$ before taking the quotient modulo $\simeq_+$, the property for a map to be of type U is not a priori compatible with precomposition of maps. This prevents us from considering $[E,A]_+$ as a functor in $E\in\mathcal Bun_*$. However, the assignment $[-,A]_+$ can act on at least three kinds of maps: $G$-equivariant morphisms of $G$-bundles, maps between codiscrete bundles,
and maps to a discrete bundle $X$, as stated by the following lemma. Note that the property (3) is counterintuitive.

\begin{lem}
\label{lemmaprecomposition}
Let $X$ and $X'$ be CW-complexes pointed in vertices,
$G$ and $G'$ countable CW-groups, $E$ a pointed $G$-bundle over $X$,  $E'$ a pointed $G'$-bundle over $X'$ and $A$ an abelian topological group. Let $\phi\colon E'\to E$ and $\psi\colon E\to A$ be morphisms in $\mathcal Bun_*$.
\begin{enumerate}
	\item When $G=G'$, if $\psi$ is of type U, then $\psi\circ\phi$ is of type U.
	\item When $G=G'$, if $\psi$ is of type wUP, then $\psi\circ\phi$ is of type wUP.
	\item When $E=X$, if $\psi$ is of type U, then $\psi\circ\phi$ is of type wUP.
	\item When $E=X$, if $\psi$ is of type wUP, then $\psi\circ\phi$ is of type wUP.
	\item When $E=G$ and $E'=G'$, if $\psi$ is of type U, then $\psi\circ\phi$ is of type wUP.
	\item When $E=G$ and $E'=G'$, if $\psi$ is of type wUP, then $\psi\circ\phi$ is of type wUP.
	\end{enumerate}
\end{lem}
\begin{proof}
\begin{enumerate}
	\item As $\psi$ is of type U, it factors through a universal bundle $\tilde E$ for $G$, as displayed:
$$\psi\colon \xymatrix{E\ar[r]^-{T}&\tilde E\ar[r]^-\kappa&A.}$$
Then $\psi\circ\phi$ factors through $\tilde E$, as displayed
$$\psi\circ\phi\colon \xymatrix{E'\ar[r]^-{\phi}&E\ar[r]^-{T}&\tilde E\ar[r]^-\kappa&A.}$$
It follows that $T\circ\phi$ is a universal map and $\psi\circ\phi$ is of type U.
	\item As $\psi$ is of type wUP, it factors through the base space $X$, as displayed:
$$\psi\colon \xymatrix{E\ar[r]^-{\pi}&X\ar[r]^-{s}&A}$$
If $\pi'\colon E'\to X'$ is the projection of $E'$ and $g\colon X\to X'$ is the geometric information of $\phi$, then $\psi\circ\phi$ factors through $\pi\circ\phi=g\circ\pi'$ as displayed:
$$\psi\circ\phi\colon \xymatrix{E'\ar[r]^-{\pi'}&X'\ar[r]^-{g}&X\ar[r]^-{s}&A.}$$
This proves that $\psi\circ\phi$ is of type $P$.
Moreover, since $\psi$ is of type wUP, the map $s$ factors up to based homotopy through a classifying space $\tilde B$ of $G$, as displayed
$$\xymatrix{X\ar[r]^-{t}&\tilde B\ar[r]^-{\overline s}&A,}$$
where $t$ is a classifying map for $E$ with respect to a universal bundle $\tilde E$.
Then the map $s\circ g$ factors up to based homotopy through the classifying space $\tilde B$ of $G$, as displayed
$$\xymatrix{X'\ar[r]^-{g}&X\ar[r]^-{t}&\tilde B\ar[r]^-{\overline s}&A,}$$
where $t\circ g$ happens to be a classifying map for $E'$ with respect to $\tilde E$. Indeed, using Proposition \ref{naturality} on $\phi:E\to E'$, we obtain pointed equivalences of $G$-bundles over $X'$:
$$
E'\cong_* g^* E\cong_* g^* t^* \tilde E\cong_*(t\circ g)^*\tilde E.
$$

It follows that $\psi\circ\phi$ is of type wUP.

\item As $\psi$ is of type U, it is based homotopic to a constant map by Remark \ref{remtrace}. Now, if $\pi'\colon E'\to X'$ is the projection of $E'$ and $g\colon X\to X'$ is the geometric information of $\phi$, then $\psi\circ\phi$ factors through the base space $X'$ as displayed:
$$\psi\circ\phi\colon \xymatrix{E'\ar[r]^-{\pi'}&X'\ar[r]^-{g}&X\ar[r]^-{\psi}&A.}$$
This proves that $\psi\circ\phi$ is of type P.
Moreover the map $\psi\circ g$ is based homotopic to the trivial map because $\psi$ is. So $\psi\circ g$ factors up to based homotopy through the classifying space $\tilde BG'$ of $G'$, as displayed
$$\xymatrix{X'\ar[r]^-{t'}&\tilde BG'\ar[r]^-{e_A}&A,}$$
It follows that $\psi\circ\phi$ is of type wUP.

\item As $\psi$ is of type wUP, it is based homotopic to a constant map by Remark \ref{remtrace}. The same argument used in (3) applies.
It follows that $\psi\circ\phi$ is of type wUP.

\item As $\psi$ is of type U, it factors through some universal map $\overline\kappa\colon\tilde E\to A$ as
$$\psi\colon \xymatrix@R=.3cm{G\ar@{^{(}->}[r]&\tilde E\ar[r]^-{\tilde\kappa}&A}.$$
As $\tilde EG$ is a $G$-bundle and $\tilde E$ is a universal bundle for $G$, there exists a $G$-equivariant pointed map $\beta\colon\tilde EG\to\tilde E$ and the inclusion $G\hookrightarrow\tilde E$ factors through it as
$$ \xymatrix@R=.3cm{G\ar@{^{(}->}[r]&\tilde EG\ar[r]^-{\beta}&\tilde E}.$$
In conclusion, $\psi\circ\phi$ factors through $\tilde EG'$, as displayed
$$\psi\circ\phi\colon \xymatrix{G'\ar@{^{(}->}[r]&\tilde EG'\ar[r]^-{\tilde E\psi}&\tilde EG\ar[r]^-{\beta}&\tilde E\ar[r]^-{\overline\kappa}&A.}$$
It follows that $\psi\circ\phi$ is of type U.
\item As $\psi$ is of type wUP, it factors through a singleton, and is therefore constant. Then $\psi\circ\phi$ is constant and is in particular of type wUP.\qedhere
\end{enumerate}
\end{proof}
The first consequence of Lemma \ref{lemmaprecomposition} is that the assignment $E\mapsto[E,A]_+$ defines a functor on the (non-full!) subcategory ${\mathcal Bun_*}_G$ of $\mathcal Bun_*$ given by of pointed $G$-bundles and $G$-equivariant maps.
\begin{prop}
Let $X$ be a CW-complex pointed in a vertex,
$G$ a countable CW-group, $E$ a pointed $G$-bundle over $X$ and $A$ an abelian countable CW-group.  Then the assignment
$$E\mapsto\mathcal [E,A]_+:={\mathcal Bun_*}(E,A)/_{\simeq_+}$$
defines a contravariant functor of $G$-bundles into abelian groups
$$[-,A]_+:{\mathcal Bun_*}_G^{\textnormal{op}}\to\mathcal Ab.$$
\end{prop}
\begin{proof}
By Lemma \ref{lemmaprecomposition}(1) and (2), plus-equivalence is compatible with precomposition.
\end{proof}

The second consequence of Lemma \ref{lemmaprecomposition} is that the assignment $E\mapsto[E,A]_+$ acts on the Nomura-Puppe sequence of Remark \ref{puppesequence} and produces a sequence of abelian groups,
$$\xymatrix{[\tilde BG,A]_+\ar[r]&[X,A]_+\ar[r]&[E^{\textnormal{c}},A]_+\ar[r]&[G^{\textnormal{c}},A]_+\ar[r]&[ \tilde\Omega X,A]_+},$$
which will be proven to be exact in the next section.
When $A=M^kZ$, this sequence becomes
$$\xymatrix{H^k(\tilde BG;Z)\ar[r]&H^k(X;Z)\ar[r]&H^k_+(E;Z)\ar[r]&H^k_+(G;Z)\ar[r]&H^k( \tilde\Omega X,Z)}.$$

\subsection{The long exact sequence}
In this section we prove that the sequence above is exact and can be extended to be a long exact sequence.
\begin{prop}
\label{characteristicexactsequence}
Let $X$ be a countable CW-complex pointed in a vertex, $G$ a countable CW-group, $E$ a pointed $G$-bundle over $X$ and $A$ an abelian countable CW-group. There is an exact sequence of abelian groups
$$\xymatrix{[\tilde BG,A]_+\ar[r]^-{-\circ t}&[X,A]_+\ar[r]^-{-\circ\pi}&[E,A]_+\ar[r]^-{-\circ\iota}&[G,A]_+\ar[r]^-{-\circ b}&[ \tilde\Omega X,A]_+}.$$
\end{prop}
We need a preliminary lemma.
\begin{lem}
\label{decomposition} 
Let $X$ be a CW-complex pointed in a vertex,
$G$ a countable CW-group,
$E$ a pointed $G$-bundle over $X$ and $A$ an abelian countable CW-group.
Every $a$-equivariant pointed morphism $\psi\colon E\to A$
such that $a\equiv_* e_A\colon G\to A$ admits a UP-decomposition.
\end{lem}
In the proof, if $\psi\colon E\to A$ is a map of bundles, we denote by $\psi^{-1}\colon E\to A$ the map obtained by inverting pointwise in $A$ the map $\psi$.
\begin{proof}[Proof of Lemma \ref{decomposition}]
We first fix a $G$-equivariant pointed map $T\colon E\to \tilde EG$, that induces a classifying map $t\colon X\to \tilde BG$. In particular there is a commutative diagram
$$\xymatrix@R=.3cm{E\ar[d]_-\pi\ar[r]^-T&\tilde EG\ar[d]^-{\tilde\pi}\\
X\ar[r]_-{t}&\tilde BG.}$$
As $a\equiv_* e_A$, by Lemma \ref{lemmaEG} $a$ extends to a pointed $a$-equivariant map  $\kappa\colon \tilde EG\to A$. Recalling the proof of Lemma \ref{lemmaEG}, we see that $\kappa$ factors through $a_* \tilde EG$, as displayed
$$\kappa\colon \xymatrix@C=1.2cm{\tilde EG\cong\tilde EG\otimes_GG\ar[r]^-{\tilde EG\otimes_Ga}&\tilde EG\otimes_GA=a_* \tilde EG\ar[r]^-{\overline\kappa}&A.}$$
The map $\overline{\kappa}$ induces pointed equivalence of $A$-bundles over $\tilde BG$ and $X$
$$(\tilde\pi,\overline\kappa)\colon a_* \tilde EG\to \tilde BG\times A\text{ and }(\pi,\overline\kappa\circ a_*T)\colon a_*E\to X\times A$$
which fits into a commutative diagram
$$\xymatrix@C=2cm@R=.5cm{E\ar[r]^-{E\otimes_Ga}\ar[d]_-T&a_*E\ar[r]^-{(\pi,\overline\kappa\circ a_*T)}_-\cong\ar[d]_-{a_*T}&X\times A\ar[d]^-{t\times A}\\
\tilde EG\ar[r]^-{\tilde EG\otimes_Ga}&a_*\tilde EG\ar[r]^-{(\tilde\pi,\overline\kappa)}_-\cong&\tilde BG\times A.
}$$
On the other hand, since $\psi$ is $a$-equivariant, it induces an $A$-equivariant map $\overline\psi\colon a_*E\to A$ and factors through $a_* E$ as displayed
$$\psi\colon \xymatrix{E\cong E\otimes_GG\ar[r]^-{E\otimes_Ga}&E\otimes_GA=a_* E\ar[r]^-{\overline\psi}&A.}$$
Using the equivalence $a_*E\cong_* X\times A$, $\psi$ factors through $X\times A$, too. The map $X\times A\to A$ occurring in such a factorization for $\psi$ is obtained by composing the inverse of $(\overline\pi,\overline\kappa\circ a_* T)$ and $\overline\psi$, which are both $A$-equivariant.
So the composite is also $A$-equivariant, and thus it must be of the form
$$s\cdot A\colon \xymatrix{X\times A\ar[r]^-{s\times A}&A\times A\ar[r]^-{\cdot}&A,}$$
where $s\colon X\to A$ is a pointed map.
There is therefore a commutative diagram
$$\xymatrix@C=2cm@R=.5cm{&&&A.\\
E\ar[r]^-{E\otimes_Ga}\ar[d]_-T\ar@/^2pc/[rrru]^-{\psi}&a_*E\ar[r]_-{(\pi,\overline\kappa\circ a_*T)}^(.6)\cong\ar[d]_-{a_*T}\ar@/^.5pc/[rru]^-{\overline\psi}&X\times A\ar[d]^-{t\times A}\ar[ru]_-{s\cdot A}&\\
\tilde EG\ar[r]^-{\tilde EG\otimes_Ga}&a_*\tilde EG\ar[r]^-{(\tilde\pi,\overline\kappa)}_-\cong&\tilde BG\times A&
}$$
There is no reason for $s\cdot A$ to factor further through $\tilde BG\times A$. Using the fact that both $\psi$ and $(s\circ\pi)^{-1}$ both factor through $a_* E$ and $X\times A$, we write a diagram similar to the one above for
$$\psi_U:=\psi\cdot(s\circ\pi)^{-1}\colon E\to A$$
instead of $\psi$.
The key fact is that $\psi_U$ is designed such that
the map $X\times A\to A$ occurring in the factorization of $\psi_U$ is now just a projection on the second factor. In particular $\psi_U$ factors through $\tilde BG\times A$, as displayed
$$\xymatrix@C=2cm@R=.5cm{&&&&A.\\
E\ar[r]^-{E\otimes_Ga}\ar[d]_-T\ar@/^2pc/[rrrru]^-{\psi_U=\psi\cdot(s\circ\pi)^{-1}}&a_*E\ar[r]_-{(\pi,\overline\kappa\circ a_*T)}^(.6)\cong\ar[d]_-{a_*T}\ar@/^.5pc/[rrru]^-{\overline\psi\cdot(s\circ\overline\pi)^{-1}}&X\times A\ar[d]^-{t\times A}\ar[rru]|-{\quad\pr_2=(s\cdot A)\cdot(s\circ\pr_1)^{-1}}&&\\
\tilde EG\ar[r]^-{\tilde EG\otimes_Ga}&a_*\tilde EG\ar[r]^-{(\tilde\pi,\overline\kappa)}_-\cong&\tilde BG\times A\ar@/_1pc/[rruu]_-{\pr_2}&&
}$$
As a consequence $\psi_U$ factors through the universal bundle $\tilde EG$ for $G$, and is therefore of type U.

We then set
$$\psi_P:=s\circ\pi\colon E\to A,$$
which is of clearly type P, and the desired factorization of $\psi$ follows:
$$\psi=(s\circ\pi)\cdot(\psi\cdot(s\circ\pi)^{-1})=\psi_P\cdot\psi_U.\qedhere$$
\end{proof}

We can now prove the proposition.
\begin{proof}[Proof of Proposition \ref{characteristicexactsequence}]
{We prove the exactness in $[X,A]_+$:} In order to show that $\im(t^*)\subset\ker(\pi^*)$, take $[s:\tilde BG\to A]_+$. As $s\circ t\circ\pi$ clearly factors through $\tilde BG$ as
$$s\circ t\circ\pi:\xymatrix{E\ar[r]^-\pi&X\ar[r]^-t&\tilde BG\ar[r]^-s&A,}$$
it is of type UP and
$$\pi^*(t^*([s]_+))=[s\circ t\circ\pi]_+=[e_A]_+,$$
namely $t^*([s]_+)\in\ker(\pi^*)$.
In order to show that $\im(t^*)\supset\ker(\pi^*)$, take $[s:X\to A]_+\in\ker(\pi^*)$. Then
$$[e_A]_+=\pi^*([s]_+)=[s\circ\pi]_+.$$
Then there exists $\psi_W:E\to A$ of type wUP and $\psi_U:E\to A$ of type U such that $s\circ\pi=\psi_W\cdot\psi_U$. As $\psi_U$ is of type U, it factors through a universal bundle $\tilde E$ for $G$ as
$$\psi\colon \xymatrix{
E \ar[r]^-T&\tilde E\ar[r]^-{\kappa}&A}.$$
As $\psi_W$ is of type wUP, in factors through $X$ as
$$\psi\colon \xymatrix{
E \ar[r]^-\pi&X\ar[r]^-{g_W}&A}.$$
Since $s\circ\pi$ and $\psi_W$ both factor through $\pi$, their algebraic information maps $a(s\circ\pi)$ and $a(\psi_W)$ are trivial. It follows that also the algebraic information $a(\psi_U)$ of $\psi_U$ is also trivial, since
$$a(\psi_U)=a((s\circ\pi)\cdot\psi_W^{-1})=a(s\circ\pi)\cdot a(\psi_W)^{-1}=e_A\cdot e_A=e_A.$$
Thus $\kappa$ factors through the quotient $\tilde B=\tilde E/G$, as
$$\psi\colon \xymatrix{
\tilde E \ar[r]^-{\tilde\pi}&\tilde B\ar[r]^-{s_U}&A.}$$
Define $g_U$ as the composite of $t$ and $s_U$, as displayed:
$$\xymatrix@C=2cm@R=.5cm{E\ar[r]_-T\ar@{->>}[d]\ar@/^1pc/[rr]^-{\psi_U}&\tilde E\ar@{->>}[d]\ar[r]&A.\\
X\ar[r]^-t\ar@/_2pc/[rru]_-{g_U}&\tilde B\ar[ru]^-{s_U}&}$$
Since $\psi_W$ is of type wUP, the map $g_W$ occurring in its factorization factors up to based homotopy through any classifying space for $G$, and in particular through $\tilde B$, as displayed
$$\xymatrix@C=2cm@R=.5cm{E\ar[r]_-\pi\ar@/^1pc/[rr]^-{\psi_W}&X\ar@{}[rd]|(.3){\simeq_*}\ar@{->>}[d]\ar[r]|-{g_W}&A.\\
&\tilde B\ar[ru]_-{s_W}&}$$
Therefore
$$s\circ\pi=\psi_U\cdot\psi_W=(g_U\circ\pi)\cdot(g_W\circ\pi)=(g_U\cdot g_W)\circ\pi.$$
and, since $\pi$ is surjective,
$$s=g_U\cdot g_W=(s_U\circ t)\cdot g_W\simeq_*(s_U\circ t)\cdot(s_W\circ t)=(s_U\cdot s_W)\circ t.$$
Finally
$$[s]_+=[(s_U\cdot s_W)\circ t]_+=t^*([s_U\cdot s_W]_+)\in\im(t^*).$$

{We prove the exactness in $[E,A]_+$.} In order to show that $\im(\pi^*)\subset\ker(\iota^*)$, take $[s\colon X\to A]_+\in[X,A]_+$ and remark that $\pi\circ\iota$ is constant. Therefore
$$\iota^*(\pi^*([s]_+))=[s\circ\pi\circ\iota]_+=[s\circ e_A]_+=[e_A]_+,$$
namely $\pi^*[s]\in\ker(\iota^*)$.
In order to show that $\ker(\iota^*)\subset\im(\pi^*)$, take $[\psi\colon E\to A]_+\in\ker(\iota^*)$. Then the algebraic information of $\psi$ is plus-equivalent to the trivial homomorphism, since
$$[e_A]_+=\iota^*([\psi]_+)=[\iota\circ\psi]_+=[a(\psi)]_+.$$
By Lemma \ref{decomposition},
there exists $\psi_P\colon E\to A$ of type P, and $\psi_U\colon E\to A$ of type U such that $\psi=\psi_P\cdot\psi_U$.
As $\psi_P$ is of type $P$, it factors through the base space $X$
$$\psi_P\colon \xymatrix{E\ar[r]^-{\pi}&X\ar[r]^-{s_P}&A.}$$
Moreover, $\psi_U\simeq_+e_A$ and
$$[\psi]_+=[\psi_P\cdot\psi_U]_+=[\psi_P]_+\cdot[\psi_U]_+=[\psi_P]_+\cdot e_A=[\psi_P]_+=[s_P\circ\pi]_+=\pi^*([s_P]_+)\in\im(\pi^*).$$

{We prove the exactness in $[G,A]_+$.} In order to show that $\im(\iota^*)\subset\ker(b^*)$, take $[\psi\colon E\to A]_+$. By Theorem \ref{classificationnice}, there exists a classifying map $b\colon \tilde\Omega X\to G$ and an $a$-equivariant map $B\colon \tilde PX\to E$. In particular there is a commutative diagram
$$\xymatrix@R=.5cm{\tilde\Omega X\ar[r]^-{\tilde\iota}\ar[d]_-b&\tilde PX\ar[d]^-B\\
G\ar[r]_-\iota&E.}$$
It follows that
$\psi\circ\iota\circ b$ factors through the universal bundle $\tilde PX$ for $\tilde\Omega X$,
$$\psi\circ\iota\circ b\colon \xymatrix@R=.08cm{\tilde\Omega X\ar@{=}[d]\ar[r]^-b&G\ar[r]^-{\iota}&E\ar@{=}[d]\ar[r]^-{\psi}&A\ar@{=}[d]\\
\tilde\Omega X\ar[r]_{\tilde\iota}&\tilde PX\ar[r]_-B&E\ar[r]^-{\psi}&A.}$$
Then $\psi\circ\iota\circ b$ is of type U and
$$b^*(\iota^*([\psi]_+))=[\psi\circ\iota\circ b]_+=[e_A]_+.$$
Therefore $\iota^*([\psi]_+)\in\ker(b^*)$.
In order to show that $\ker(b^*)\subset\im(\iota^*)$, take $[a\colon G\to A]_+\in\ker(b^*)$. Then
$$[e_A]_+=b^*([a]_+)=[a\circ b]_+,$$
and, by Proposition \ref{degeneratecase}(2), $a\circ b\equiv_* e_A$.
There are then pointed equivalences of $A$-bundles over $X$
$$a_* E\cong_*a_* b_*\tilde PX\cong_*(a\circ b)_*\tilde PX\cong_*{e_A}_*\tilde PX=X\times A$$
that we use to construct the map $\psi\colon E\to a_*E\cong X\times A\to A$,
whose algebraic information can be computed to be $a\colon G\to A$. This means that
$$\im(\iota^*)\ni\iota^*([\psi]_+)=[\psi\circ\iota]_+=[a(\psi)]_+=[a]_+.\qedhere$$
\end{proof}

If we specialize the exact sequence to the case $A=M^kZ$ we obtain the following long exact sequence, which can be regarded as a dual of the homotopy long exact sequence of a fibration.

\begin{teorema}
\label{longexactsequence}
Let $X$ be a connected countable CW-complex (pointed in a vertex), $G$ a topological group whose homotopy groups are countable, $E$ a $G$-bundle over X, and $Z$ a countable abelian group.
There is a long exact sequence of abelian groups
$$\xymatrix{\dots\ar[r]&H^k(X;Z)\ar[r]&H^k_+(E;Z)\ar[r]&H^{k+1}(BG;Z)\ar[r]^-{\chi_{k+1}}&H^{k+1}(X;Z)\ar[r]&\dots}$$
where the map $\chi_k$ is induced by the classifying map of $E$.
\end{teorema}

\begin{proof} Let $t:X\to\tilde BG^c$ be a classifying map for $E^c$. We first notice that by Theorem \ref{pointedclassificationnice} there exists a map of topological groups $\tilde\Omega t\colon \tilde\Omega X\to G^{\text{c}}$ such that $E^{\textnormal{c}}\cong_* (\tilde\Omega t)_*\tilde PX$.
By applying Proposition \ref{characteristicexactsequence} to $E^{\textnormal{c}}$ with $k>0$ and $A:=M^kZ$ we obtain a sequence of abelian groups with five terms:
$$\xymatrix{
[\tilde BG^c,M^kZ]_+\ar[r]^-{-\circ t}&[X,M^kZ]_+\ar[r]&[E^{\text{c}},M^kZ]_+\ar[r]&[G^{\text{c}},M^kZ]_+\ar[r]^-{-\circ\tilde\Omega t}&[ \tilde\Omega X,M^kZ]_+.}$$
By Lemma \ref{connectingmap}, there is a commutative diagram
$$\xymatrix@C=2.cm@R=.5cm{\ar[d]_-\cong[G^c,M^kZ]_+\ar[r]^-{-\circ\tilde\Omega t}&\ar[d]_-\cong[\tilde\Omega X,M^kZ]_+\\
[\tilde BG^c,M^{k+1}Z]_+\ar[d]_-{\cong}\ar[r]_-{-\circ t}&[X,M^{k+1}Z]_+\ar[d]_-\cong.\\
H^{k+1}(\tilde BG^c;Z)\ar[r]_-{\chi_{k+1}}&H^{k+1}(X;Z).\\
}$$
Therefore the connecting map to the right in degree $k$,
$$-\circ\tilde\Omega t\colon [G^{\text{c}},M^kZ]_+\to[ \tilde\Omega X,M^kZ]_+,$$
can be identified with the connecting map map to the left in degree $k+1$,
$$-\circ t\colon [\tilde BG,M^{k+1}Z]_+\to[ X,M^{k+1}Z]_+,$$
which can further be identified with the map induced in cohomology by $t$ in degree $k+1$,
$$\chi_{k+1}\colon H^{k+1}(\tilde BG;Z)\to H^{k+1}(X;Z).$$
$$\xymatrix@R=.5cm{
[\tilde BG,M^kZ]_+\ar[r]\ar[d]_-{\cong}&\ar[d]_-{\cong}[X,M^kZ]_+\ar[r]&\ar[d]_-{\cong}[E^{\text{c}},M^kZ]_+\ar[r]&\ar[d]_-{\cong}[G^{\text{c}},M^kZ]_+\ar[r]&\ar[d]_-{\cong}[ \tilde\Omega X,M^kZ]_+\\
H^k(\tilde BG;Z)\ar[r]&H^k(X;Z)\ar[r]&H_+^k(E;Z)\ar[r]&H^{k+1}(\tilde BG;Z)\ar[r]&H^{k+1}(X;Z)
}$$
Starting with $k=0$ we can extend the sequence to the right infinitely many times, and obtain the desired long exact sequence.
\end{proof}

We specialize the sequence above to two extreme cases and use Corollary \ref{degeneratecasegroupscor} to simplify it.
\begin{ex}
Let $X$ be a connected countable CW-complex (pointed in a vertex), $G$ a topological group whose homotopy groups are countable, $E$ a $G$-bundle over X, and $Z$ a countable abelian group.
\begin{itemize}
	\item[>>] If $E=X\times G$ is the trivial bundle,
	the long exact sequence of Theorem \ref{longexactsequence} splits
$$\xymatrix@C=.4cm{\dots\ar[r]&H^k(X;Z)\ar[r]&H^k(X;G)\oplus H^{k+1}(\tilde BG;Z)\ar[r]&H^{k+1}(\tilde BG;Z)\ar[r]^-{0}&H^{k+1}(X;Z)\ar[r]&\dots}.$$
	\item[>>] If $E=EG$ is a universal bundle,
 the long exact sequence of Theorem \ref{longexactsequence} degenerates to
		$$\xymatrix@C=.8cm{\dots\ar[r]&H^k(X;Z)\ar[r]&0\ar[r]&H^{k+1}(\tilde BG;Z)\ar[r]^-{\cong}&H^{k+1}(X;Z)\ar[r]&\dots}.$$
\end{itemize}
\end{ex}

\subsection{Applications}
We collect here some interesting consequences of Theorem \ref{longexactsequence}, that allow us to compare the plus-cohomology groups with ordinary cohomology groups and with an instance of the twisted cohomology from \cite{nss}. The detailed proofs can be found in the author's thesis \cite{rovellitesi}.

\vphantom{}

For any $G$-bundle $E$, the fiber inclusion $\iota\colon G\hookrightarrow E$ is a Hurewicz cofibration. By \cite[Theorem 3.2.1]{piccinini}, if we let the ordinary cohomology groups act on the cofiber sequence of $\iota\colon G\hookrightarrow E$,\linebreak$\xymatrix@C=.5CM{G\ar[r]^-\iota&E\ar[r]&\textrm{cof}(\iota)\ar[r]&\Sigma  G}$,
we obtain a cohomology long exact sequence
$$\xymatrix@C=1cm{
\dots\ar[r]&H^k(\textrm{cof}(\iota);Z)\ar[r]_-{}&H^k(E;Z)\ar[r]^-{H^k(\iota;Z)}&H^{k}(G;Z)\ar[r]&H^{k+1}(\textrm{cof}(\iota);Z)\ar[r]&\dots
}.$$
This sequence can be compared to the sequence of Theorem \ref{longexactsequence}.

\begin{prop}[\cite{rovellitesi}]
Let $X$ be a connected countable CW-complex (pointed in a vertex), $G$ a topological group whose homotopy groups are countable, $E$ a $G$-bundle over X, and $Z$ a countable abelian group.
The comparison map of Remark \ref{trace} induces a morphism of long exact sequences of abelian groups
$$\xymatrix@R=.5cm{\dots\ar[r]&H^k(X;Z)\ar[r]\ar[d]_-{H^k(\textrm{cof}(\iota)\to X;Z)}&H^k_+(E;Z)\ar[r]\ar[d]&H^{k+1}(BG;Z)\ar[r]^-{\chi_{k+1}}\ar[d]&H^{k+1}(X;Z)\ar[r]\ar[d]^-{H^{k+1}(\textrm{cof}(\iota)\to X;Z)}&\dots\\
\dots\ar[r]&H^k(\textrm{cof}(\iota);Z)\ar[r]_-{}&H^k(E;Z)\ar[r]_-{H^k(\iota;Z)}&H^{k}(G;Z)\ar[r]&H^{k+1}(\textrm{cof}(\iota);Z)\ar[r]&\dots\\
}$$
\end{prop}

It was pointed out by Jeffrey Carlson (personal communication, November 15, 2016) that there is a different way to produce a long exact sequence involving the characteristic map. Indeed, the cofiber sequence $\xymatrix@C=.5CM{X\ar[r]^-t&BG\ar[r]&\textrm{cof}(t)\ar[r]&\Sigma  X}$
induces a long exact sequence
$$
\xymatrix@C=.7cm{\dots\ar[r]&H^{k}(\textrm{cof}(t);Z)\ar[r]&H^k(BG;Z)\ar[r]^-{t^*}&H^k(X;Z)\ar[r]&H^{k+1}_+(\textrm{cof}(t);Z)\ar[r]&\dots}.
$$
The fact that the plus-cohomology group $H^k_+(E;Z)$ and the ordinary cohomology group\linebreak$H^{k+1}(\textrm{cof}(t);Z)$ of the cofiber $\textrm{cof}(t)$ of the classifying map $t:X\to BG$ both fit into long exact sequences suggests that one can hope for a comparison. In \cite{rovellitesi}, a comparison map
$$H^{k+1}(\textrm{cof}(t);Z)\to H^k_+(E;Z)$$
is constructed, and proven to be an isomorphism using Theorem \ref{longexactsequence} together with standard techniques of homological algebra. This is summarized by the following theorem.

\begin{teorema}[\cite{rovellitesi}]
Let $X$ be a connected countable CW-complex (pointed in a vertex), $G$ a topological group whose homotopy groups are countable, $E$ a $G$-bundle over X, and $Z$ a countable abelian group.
There is an isomorphism
$$H_+^k(E;Z)\cong H^{k+1}(\textrm{cof}(t);Z)$$
that induces an isomorphism of long exact sequences of abelian groups
$$\xymatrix@R=.5cm{\dots\ar[r]&H^k_+(X;Z)\ar[d]_-{\cong}\ar[r]^-{\pi^*}&H^k_+(E;Z)\ar[r]^-{\iota^*}\ar[d]_-\cong&H^{k}_+(G;Z)\ar[r]^-{t^*}\ar[d]_-\cong&H^{k+1}_+(X;Z)\ar[r]\ar[d]^-{\cong}&\dots\\
\dots\ar[r]&H^k(X;Z)\ar[r]_-{}&H^k(\textrm{cof}_\star(t);Z)\ar[r]_-{}&H^{k+1}(BG;Z)\ar[r]_-{t^*}&H^{k+1}(X;Z)\ar[r]&\dots\\
}$$
\end{teorema}

This isomorphism gives an interpretation of the plus-cohomology groups in purely homotopical terms, and does not involve any assumption of countability for the (CW-structure of the) structure group $G$. The geometric definition in terms of plus-cohomology groups, however, highlights better the role played by the bundle structure. For instance, the plus-cohomology groups of a bundle $E$ can be used to detect
the number of structures of a fixed type on a bundle, e.g. the number of reductions of the structure group to a specific one, as we now explain.

\vphantom{}

Given $X$ a connected countable CW-complex, $G$ a topological group whose homotopy groups are countable, $E$ a $G$-bundle over X, $Z$ a countable abelian group, and $c\in H^k(BG;Z)$ a universal characteristic class, one can be interested in counting the number of $\hat G(c)$-reductions of $G$.
For instance, when $c:=w_2\in H^2(BSO(n);\mathbb Z/2)$ is the second Stiefel-Whitney class, and $E$ is an $SO(n)$-bundle $E$ over $X$, the question translates to counting the number of spin structures on $E$ (up to a suitable equivalence relation). 
Combining Theorems \ref{groupreductioncor} and \ref{retractioncor}, we see that there exists a $\hat G(c)$-reduction for a $G$-bundle $E$ if and only if the fiber
$$H^{k-1}_+(E;Z)_c:=\textrm{fib}_c\left(\iota^*\colon H^{k-1}_+(E;Z)\to H^{k-1}_+(G;Z)\cong H^k(BG;Z)\right)$$
of $\iota^*$ over $c\in H^k(BG;Z)$ is not empty. This suggests that such fiber could be an indication of how many such reductions there exist. 
On the other hand, in the literature the $\hat G(c)$-reductions of a bundle $E$ are often enumerated by means of $\pi_0(\Gamma^h(Bj;t))$, i.e., the set of connected components of the space of homotopy liftings of the classifying map $t\colon X\to BG$ of $E$ along the map $Bj\colon B\hat G(c)\to BG$. This is an instance of twisted cohomology (as in \cite{nss}).
The link between twisted cohomology and plus-cohomology is made precise by the following theorem.

\begin{teorema}[\cite{rovellitesi}] 
\label{comparisontwistedcohomology}
Let $X$ be a connected countable CW-complex, $G$ a topological group whose homotopy groups are countable, $E$ a $G$-bundle over X, $Z$ a countable abelian group, and $c\in H^k(BG;Z)$ a universal characteristic class.
There is a transitive action of $H^{k-1}(BG;Z)$ on the twisted cohomology $\pi_0(\Gamma^h(Bj;t))$, such that
$$H^{k-1}_+(E;Z)_c\cong \pi_0(\Gamma^h(Bj;t))/_{H^{k-1}(BG;Z)}.$$
\end{teorema}
The isomorphism is built by making explicit use of the construction of the plus-cohomology groups. The action, together with the isomorphism, provide useful information on the cardinality of the twisted cohomology, namely on the number of structures of a specific kind on a bundle.

\vphantom{}

Very often the action is trivial, and the fiber $H^{k-1}_+(E;Z)_c$ therefore coincides with the twisted cohomology $ \pi_0(\Gamma^h(Bj;t))$. This happens, for instance, when the cohomology of $BG$ vanishes in certain degrees. Here is an application that exploits this idea.

\begin{ex}
Recall from Construction \ref{whiteheadtower} the Whitehead tower $(O(n)^{(k)})_{k\ge0}$ of the orthogonal group $O(n)$.
We first observe that, by the universal coefficient theorem for cohomology and the Hurewicz Theorem,
$$\begin{array}{rcl}
H^{k-1}(BO(n)^{(k)};\mathbb Z/2)&\cong&\textrm{Hom}(H_{k-1}(BO(n)^{(k)};\mathbb Z),\mathbb Z/2)\\
&\cong& \textrm{Hom}(\pi_{k-1}(BO(n)^{(k)}),\mathbb Z/2)\\
&\cong& \textrm{Hom}(0,\mathbb Z/2)\cong0.\\
\end{array}$$
The induced map $H^k(BO(n);\mathbb Z/2)\to H^k(BO(n)^{(k)};\mathbb Z/2)$ allows us to think of the $k$th universal Stiefel-Whitney class $w_k\in H^k(BO(n);\mathbb Z/2)$ as a characteristic class living in $H^k(BO(n)^{(k)};\mathbb Z/2)$. Now, given an $O(n)$ bundle $E$ over $X$, by Theorem \ref{comparisontwistedcohomology}, the (possibly empty) fiber $H^0_+(E;\mathbb Z/2)_{w_1}$ counts the $SO(n)$-reductions of $E$, namely the orientations of $E$. If this set is not empty, there exists a bundle $E^{(1)}$ such that $E\cong E^{(1)}\otimes_{SO(n)}O(n)$. By Theorem \ref{comparisontwistedcohomology}, the (possibly empty) fiber $H^1_+(E^{(1)};\mathbb Z/2)_{w_2}$ coincides with twisted cohomology, and counts the $Spin(n)$-reductions of $E^{(1)}$, namely the spin structures of $E^{(1)}$. If this set is non-empty, there exists a bundle $E^{(2)}$ such that $E^{(1)}\cong E^{(2)}\otimes_{Spin(n)}SO(n)$. Inductively, if there exists an $O(n)^{(k)}$-bundle $E^{(k)}$ such that $E^{(k-1)}\cong E^{(k)}\otimes_{G^{(k-1)}}G^{(k)}$, by Theorem \ref{comparisontwistedcohomology}, the (possibly empty) fiber $H^{k-1}_+(E^{(k-1)};\mathbb Z/2)_{w_k}$ counts the $G^{(k)}$-reductions of $E^{(k-1)}$.
\end{ex}

\section{Appendix: Proof of Proposition \ref{totalspaceCWcomplex}}
In this section we prove Proposition \ref{totalspaceCWcomplex}, which is needed for Propositions \ref{fibersequence1} and \ref{propertiestypeofmorphisms}. The proof uses certain properties of compactly generated weakly Hausdorff spaces. Recall that a space $X$ is \textbf{weakly Hausdorff} if for every compact Hausdorff space $K$ and continuous map $f\colon K\to X$ the image $f(K)$ is closed. A function $f\colon X\to Y$ between topological spaces is \textbf{$k$-continuous} if for any compact Hausdorff space $C$ the precomposition $f\circ j\colon C\to Y$ of $f$ with a continuous map $j\colon C\to X$ is continuous (cf. \cite[Proposition 1.11]{strickland}).
The space $X$ is \textbf{compactly generated} if every k-continuous function $f\colon X\to Y$ is continuous.
It was proven in \cite{strickland} that compactly generated weakly Hausdorff spaces and continuous maps form a category $CGWH$ of convenient topological spaces, i.e., the category $CGWH$ contains all CW-complexes, is bicomplete and is cartesian closed.
\begin{rem}
\label{limitsandcolimitsinCGWH}
The category $CGWH$ of compactly generated weakly Hausdorff spaces is reflective in the category $CG$ of compactly generated topological spaces, and the category $CG$ is coreflective in the category $\mathcal Top$ of all spaces. The left adjoint $h\colon CG\to CGWH$ (described in \cite[Proposition 2.22]{strickland}) and the right adjoint $k\colon \mathcal Top\to CG$ (defined in \cite[Definition 1.1]{strickland}) are called \textbf{$h$-fication} and \textbf{$k$-fication} respectively. These functors help construct (co)limits in $CGWH$.
\begin{enumerate}
\item Colimits in $CGWH$ are constructed by $h$-fying the colimits in $\mathcal Top$, as proven in \cite[Corollary 2.23]{strickland}. In formulas,
$$\textrm{colim}_i^{CGWH}C_i\cong h(\textrm{colim}_i^{\mathcal Top}C_i).$$
In particular, if a colimit in $\mathcal Top$ of a diagram of CGWH spaces is a CGWH space, then it coincides with the colimit in $CGWH$.
\item Products in $CGWH$ are constructed by k-fying the products in $\mathcal Top$, because the inclusion of $CGWH$ in $CG$ is a right adjoint and products in $CG$ are constructed by k-fying the products in $\mathcal Top$ \cite[Proposition 2.4]{strickland}. In formulas,
$$C\times_{CGWH}C'\cong k(C\times_{\mathcal Top}C').$$
In particular, if a product in $\mathcal Top$ of two CGWH spaces is a CGWH space, then it coincides with the product in $CGWH$.
\end{enumerate}
\end{rem}
We now prove two closure properties of CGWH spaces.
\begin{lem}
\label{CGWH}
\begin{enumerate}
\item Every open subset of a CW-complex is a CGWH space.
\item Every fiber bundle whose base space is a CW-complex and whose fiber is a CW-complex is a CGWH space.
\end{enumerate}
\end{lem}
\begin{proof}[Proof of Lemma \ref{CGWH}]
We prove (1). As CW-complexes are normal spaces \cite[Proposition A.3.]{hatcher}, they are in particular regular. Whence, by \cite[Lemma 2.1(a)]{munkres}, every open neighborhood of a point contains a closed neighborhood of the same point. In particular, for every $U$ open subset of a CW-complex, each point of $U$ has a closed neighborhood in $U$. By  \cite[Chapter 5, Problem 1]{may99} this is enough for $U$ to be a CGWH space.

We prove (2). Let $E$ be a bundle over the countable CW-complex $X$ and with fiber the countable CW-complex $F$. We prove that $E$ is Hausdorff. Let $e,e'$ be two distinct points in $E$. If $\pi(e)\neq\pi(e')$, then there exist $U$ and $U'$ disjoint neighborhoods of $\pi(e)$ and $\pi(e')$ in $X$. So $E|_U$ and $E|_{U'}$ are disjoint neighborhoods of $e$ and $e'$. If $\pi(e)=\pi(e')$, then there exist $W$ a trivialization open that contains $\pi(e)$. As $E|_W\cong W\times F$, which is a product of Hausdorff spaces, the points $e$ and $e'$ can be separated in the open $E|_W$ by two disjoint open neighborhoods.

We prove that $E$ is compactly generated. Let $f\colon E\to Y$ be a k-continuous function between topological spaces.
Let $e$ be a point of $E$, and $U$ a trivialization open containing $\pi(e)$. The space $E|_U\cong U\times F$ is (homeomorphic to) an open subset of $X\times F$, which is a product of countable CW-complexes, and therefore a CW-complex. By (1), $E|_U$ is a CGWH space, and $f|_{E|U}$ is k-continuous, so $f|_{E|U}$ is continuous. Since the open subsets of the form $E|_U$ form a cover of $E$, it follows that $f$ is continuous, and $E$ is a CGWH space.
\end{proof}

We can now prove Proposition \ref{totalspaceCWcomplex}.
\begin{proof}[Proof of Proposition \ref{totalspaceCWcomplex}]
If $\sigma$ is a closed cell of $X$ and $U$ is a trivializing open subset of $X$, the following spaces are CGWH: the spaces $X$, $F$ and $\sigma$, as they are CW-complexes, the space $E$, as proven in Lemma \ref{CGWH}(2), the spaces $X\times F$ and $\sigma\times F$, as they are products of countable CW-complexes, and therefore CW-complexes by \cite[Theorem A.6]{hatcher},
the spaces $U$, $U\cap\sigma$, $U\cap\sigma\times F$, and $U\times F$, as they are open subsets of CW-complexes, and therefore CGWH spaces by Lemma \ref{CGWH}(1), and the space $E|_U$, as it is homeomorphic to $U\times F$, and therefore a CW-complex.

By Remark \ref{limitsandcolimitsinCGWH}(1), colimits of diagrams involving these spaces coincide, be they taken in the category $\mathcal Top$ of all spaces or in the category $CGWH$ of CGWH spaces.

We prove that the topological space $E$ is a colimit of its restrictions $E|_{\sigma}$ to closed cells $\sigma$ of $X$. If $\sigma$ varies over the closed cells of $X$ (and inclusions of those) and $U$ varies over the trivialization open subsets of $X$ for $E$ (and inclusion of those), we have the following homeomorphisms
$$\begin{array}{cclr}
E&\cong&\textrm{colim}_U[E|_U]&\text{$\{E|_U\}$ is an open cover of $E$}\\
&\cong&\textrm{colim}_U[U\times F]&\text{$E$ is locally trivial over $U$}\\
&\cong&\textrm{colim}_U[(\textrm{colim}_{\sigma}[U\cap\sigma])\times F]&\text{$X$ has the weak topology and $U$ is an open subset}\\
&\cong&\textrm{colim}_U[\textrm{colim}_{\sigma}[(U\cap\sigma)\times F]]&\text{$-\times F$ is a left adjoint (in $CGWH$)}\\
&\cong&\textrm{colim}_{\sigma}[\textrm{colim}_U[(U\cap\sigma)\times F]]&\text{Fubini Theorem \cite[IX.8]{maclane}}\\
&\cong&\textrm{colim}_{\sigma}[(\textrm{colim}_{U}[U\cap\sigma])\times F]&\text{$-\times F$ is left adjoint (in $CGWH$)}\\
&\cong&\textrm{colim}_{\sigma}[\sigma\times F]&\text{$\{U\cap\sigma\}$ is an open cover of $\sigma$.}\\
\end{array}$$
Now, every piece $\sigma\times F$ is a countable CW-complex, because it is the product of two countable CW-complexes. We therefore conclude that $E$ is a countable CW-complex.
\end{proof}

\bibliography{references}
\bibliographystyle{alpha}
\end{document}